\documentclass[preprint,1p,10pt]{elsarticle}
\usepackage{lineno,hyperref,latexsym,amsmath,amssymb,graphicx}
\usepackage{amsthm}

\journal{Linear Algebra and its Applications}
\usepackage[all,web]{xy}
\usepackage{color}
\usepackage{hyperref}

\usepackage{xcolor}

\def\xyellowspace{%
  \sbox0{\colorbox{yellow}{\strut\ }}
  \dimen0=\wd0\relax
  \hskip0pt\cleaders\box0\hskip\dimen0\hskip0pt}

\begingroup
\catcode`\ =\active%
\gdef\makeyellowspace{\let \xyellowspace\catcode`\ =\active}%
\endgroup

\def\?#1{\colorbox{yellow}{\strut#1}}

\def\urlfont{\DeclareFontFamily{OT1}{cmtt}{\hyphenchar\font='057}
              \normalfont\ttfamily \hyphenpenalty=10000}

\DeclareFontFamily{OT1}{rsfs10}{}
\DeclareFontShape{OT1}{rsfs10}{m}{n}{ <-> rsfs10 }{}
\DeclareMathAlphabet{\mathscript}{OT1}{rsfs10}{m}{n}

\DeclareMathOperator{\im}{Im}       
\DeclareMathOperator{\Spec}{Spec}   
\DeclareMathOperator{\Hom}{Hom}     
\DeclareMathOperator{\Tors}{Tors}    
\DeclareMathOperator{\Pic}{Pic}     
\DeclareMathOperator{\Cl}{Cl}       
\DeclareMathOperator{\rk}{rk}       
\DeclareMathOperator{\Mov}{Mov}     
\DeclareMathOperator{\diag}{diag}   
\DeclareMathOperator{\fan}{fan}     
\DeclareMathOperator{\lcm}{lcm}     
\DeclareMathOperator{\HNF}{HNF}     
\DeclareMathOperator{\SNF}{SNF}     

\def \d{\delta }

\def \s{\sigma }

\def \Si{\Sigma }

\def \q{\mathbf{q}}

\def \u{\mathbf{u}}
\def \v{\mathbf{v}}
\def \n{\mathbf{n}}
\def \w{\mathbf{w}}
\def \t{\mathbf{t}}
\def \z{\mathbf{z}}
\def \1{\mathbf{1}}
\def \0{\mathbf{0}}

\def\P{{\mathbb{P}}}

\def\p2{\mathbb{P}^2}
\def\p3{\mathbb{P}^3}
\def\p4{\mathbb{P}^4}

\def\rk{\operatorname{rk}}

\def\GL{\operatorname{GL}}
\def\SL{\operatorname{SL}}

\def\Z{\mathbb{Z}}

\def\C{\mathbb{C}}
\def\R{\mathbb{R}}
\def\M{\mathbf{M}}
\def\Q{\mathbb{Q}}
\def\N{\mathbb{N}}

\def\SF{\mathcal{SF}}

\def\G{\mathcal{G}}

\begin{document}

\begin{frontmatter}

\title{$\Z$--linear Gale duality and\\ poly weighted spaces (PWS)}

\author[rvt]{Michele Rossi\corref{cor1}\fnref{fn1,fn2}}
 \ead{michele.rossi@unito.it}

 \author[rvt]{Lea Terracini\corref{cor1}\fnref{fn1}}
 \ead{lea.terracini@unito.it}

 \cortext[cor1]{Corresponding author}

 \address[rvt]{Dipartimento di Matematica, Universit\`a di Torino,
 via Carlo Alberto 10, 10123 Torino }

 \fntext[fn1]{The authors were partially supported by the MIUR-PRIN 2010-11 Research Funds ``Geo\-metria  delle Variet\`{a} Algebriche''.}
 \fntext[fn2]{The author is supported by the I.N.D.A.M. as a member of the G.N.S.A.G.A.}

\newtheorem{theorem}{Theorem}[section]
\newtheorem{proposition}[theorem]{Proposition}
\newtheorem{criterion}[theorem]{Criterion}
\newtheorem{thm-def}[theorem]{Theorem--Definition}
\newtheorem{corollary}[theorem]{Corollary}
\newtheorem{conjecture}[theorem]{Conjecture}
\newtheorem{lemma}[theorem]{Lemma}
\newtheorem{sublemma}[theorem]{Sub--lemma}
\newtheorem{remark}[theorem]{Remark}
\newtheorem{remarks}[theorem]{Remarks}
\newtheorem{example}[theorem]{Example}
\newtheorem{examples}[theorem]{Examples}
\newtheorem{definition}[theorem]{Definition}
\newtheorem{algorithm}[theorem]{Algorithm}


%
\newcommand{\oneline}{\vskip12pt}
\newcommand{\halfline}{\vskip6pt}
%

\newcommand{\cy}{Ca\-la\-bi--Yau }
\newcommand{\ka}{K\"{a}hler }
\newcommand{\hookuparrow}{\cup\hskip-4.15pt{ }^{{}^{\textstyle\uparrow}}}
\newcommand{\mapsfrom}{\mathrel{\reflectbox{\ensuremath{\mapsto}}}}
\newcommand{\longmapsfrom}{\mathrel{\reflectbox{\ensuremath{\longmapsto}}}}


\begin{abstract} The present paper is devoted to discussing Gale duality from the $\Z$--linear algebraic point of view. This allows us to isolate the class of $\Q$--factorial complete toric varieties whose class group is torsion free, here called \emph{poly weighted spaces} (PWS). It provides an interesting generalization of weighted projective spaces (WPS).
\end{abstract}


\end{frontmatter}


\tableofcontents
\section*{Introduction}

The present paper is the first part of a longstanding tripartite
study aimed to realizing, for $\Q$--factorial
 projective toric varieties, a birational classification inspired by what V.~Batyrev did in 1991 for smooth complete toric varieties \cite{Batyrev91}. This paper is then devoted to giving the necessary algebraic and geometric background. In particular we generalize many properties of weighted projective spaces (WPS, for short), as studied in  \cite{RT-WPS}, to the case of $\Q$--factorial and complete toric varieties of Picard number (in the following called \emph{rank}) greater than 1.\\
 The main subject of this paper will be the  study of what we call \emph{$\Z$--linear Gale Duality} and its applications to the geometry of toric varieties. Gale duality (or \emph{Gale transform}) has the origin of its name to the article \cite{Gale}, by D.~Gale, devoted to the study of polytopes, although the geometry connected with Gale duality is actually much older. D.~Eisenbud and S.~Popescu in \cite[\S~1]{EP} so write: ``Perhaps the first result that belongs to the development of the Gale transform is the theorem of Pascal (from his ``Essay Pour Les Coniques'', from 1640,[...])''. Gale duality has found its main application in the classification of polytopes \cite[\S~4, \S~6]{Ewald} and probably for this reason it has been essentially developed on the real field $\R$. The pivotal paper in introducing Gale duality for toric varieties was \cite{Oda-Park}, where T.~Oda and H.S.~Park employed Gale duality to study the toric geometric implication of the \emph{Gel'fand--Kapranov--Zelevinsky decomposition} of the \emph{secondary polytope} of a convex polytope \cite{GKZ},\cite{GKZ2}. D.A.~Cox, J.B.~Little and H.K.~Schenck dedicated an entire chapter, of their comprehensive book on toric varieties, to the discussion of Gale duality and the secondary fan \cite[Ch. 15]{CLS}.\\
 \noindent Largely inspired by the latter, in this paper we consider Gale duality, as defined in \cite[\S~14.3]{CLS}, from the $\Z$--linear point of view (see \S~\ref{ssez:Gale}), which briefly is a duality, well defined up to left multiplication by a unimodular integer matrix, between a \emph{fan matrix} of a toric variety, defined by taking the primitive generators of all the fan's rays, and a \emph{weight matrix}: in the easiest case of a WPS, the latter is precisely the row vector of weights. In the more general case of a $\Q$--factorial complete toric variety, the weight matrix admits so many rows as the rank of the considered variety. Such a weight matrix should be compared with the \emph{combined weight systems} (CWS) of M.~Kreuzer and H.~Skarke \cite{KS}: we believe that the Gale duality gives an algebraic background in which a CWS could be thought of. \\
 For an extension of Gale duality to maps of finitely generated abelian groups see \cite{BCS}.\\

The core of this paper is  \S~\ref{sez:motivazioni}, which contains the motivations of our work and the main geometric results. It starts with a contextualization of the $\Z$--linear algebraic Gale duality in the toric geometric setup, allowing to associate to a  complete $\Q$-factorial toric variety $X$ a fan matrix $V$ and a weight matrix  $Q$. The first important geometric result is Theorem \ref{thm:pi11=Tors}, which gives a linear algebraic characterization of the torsion of the class group $\Cl(X)$ in terms of the fan matrix $V$ of $X$ (see Prop. \ref{prop:PWS}).  In particular we see that the class group $\Cl(X)$ is a free abelian group exactly when we can recover the fan matrix $V$ as a Gale dual of the weight matrix $Q$;
this gives a first motivation to isolate the class of complete $\Q$--factorial toric varieties $X$ such that $\Tors(\Cl(X))=0$, called \emph{poly weighted spaces} (PWS, for short) (see  Def.~\ref{def:PWS} and Prop. \ref{prop:PWS}). In Remark~\ref{rem:motivazioni} we list a number of motivations for such an appellation some of which can be briefly summarized as follows:
\begin{enumerate}
  \item[1.] for rank 1, a PWS is a WPS,
  \item[2.] the weight matrix of a PWS $X$ allows one to completely determine the action of $\Hom(\Cl(X),\C^*)$ describing $X$ as a Cox geometric quotient, just like the weight vector of a WPS does,
  \item[3.] Gale duality generalizes to higher rank the well known characterization of a fan matrix of a PWS by means of the weights, as described e.g. in \cite[Thm.~3]{RT-WPS}.
\end{enumerate}

Actually the main motivation to the study of PWS is given in \S~\ref{ssez:finale} and will be the subject of the forthcoming paper \cite{RT-QUOT}: namely the Batyrev--Cox and Conrads result \cite[Lemma~2.11]{BC}, \cite[Prop.~4.7]{Conrads} describing a $\Q$--factorial complete toric variety of rank 1 as a finite quotient of a WPS (in the literature called a \emph{fake WPS}), can be generalized to higher rank by replacing the covering WPS with a suitable PWS: therefore every $\Q$--factorial complete toric variety is a \emph{fake PWS} i.e. a finite quotient of a PWS.

\noindent An important result of \S~\ref{sez:motivazioni} is given by Theorem~\ref{thm:generazione}, which is the generalization, to higher rank, of \cite[Prop.~8]{RT-WPS} and exhibits the bases of the subgroup of Cartier divisors inside the free group of Weil divisors and of the Picard subgroup inside the class group, respectively, for every PWS. Example \ref{ex} gives an account of all the described techniques.\\
 Section \ref{sez:GaleDuality} is devoted to present the algebraic background.
\noindent When considered from the $\Z$--linear point of view, Gale duality turns out to relate properties of every submatrix of a fan matrix, defined by the choice of a subset of fan's generators, with properties of the \emph{complementary} submatrix, in the sense of (\ref{Gale_complementare}), of a weight matrix: here the main results are Theorem~\ref{thm:Gale2} and Corollary \ref{cor:Gale-on-det}. As an example of an application of these results, one can control the singularities of a given projective $\Q$--factorial toric variety $X$, by looking at a suitable class of simplicial cones in $\Cl(X)\otimes \R$ containing the K\"{a}hler cone of $X$ (this particular aspect will be deeply studied and applied in the forthcoming paper \cite{RTclass}). In section \ref{sez:GaleDuality}, $\Z$--linear Gale duality will be studied from the pure linear algebraic point of view, without any reference to the geometric properties of the underlying toric varieties, although speaking about \emph{$F$--matrices} and \emph{$W$--matrices} (see \S~\ref{ssez:F&W}) we are clearly referring to possible fan matrices and weight matrices, respectively. Further significant results of \S~\ref{sez:GaleDuality} are
\begin{itemize}
  \item Theorem \ref{thm:Wpositiveness} stating the equivalence, via Gale duality, between \emph{$F$--complete matrices} and \emph{$W$--positive matrices} (see Definition \ref{def:Fcompletezza&Wpositività}), geometrically meaning that a PWS is characterized by the  choice of a positive weight matrix,
  \item Theorem \ref{thm:riduzione} generalizing to any rank the well known reduction process for the weights of a WPS.  It gives an easy interpretation by Gale duality: briefly a weight matrix is \emph{reduced} if and only if a (and every) Gale dual matrix of it is a matrix whose columns admit only coprime entries;
  \item Theorem \ref{thm:REF} whose geometric meaning is that a weight matrix of a complete $\Q$--factorial toric variety can always be assumed to be a positive matrix in row echelon form.
\end{itemize}

Section \ref{sez:proofs} contains the proofs of results presented in Section \ref{sez:motivazioni}. It also provides some constructive techniques for dealing with fan and weight matrices. For example
 Proposition~\ref{prop:QdaVeviceversa} gives a procedural recipe for Gale duality: a Gale dual matrix $Q$ of a fan matrix $V$ can be recovered by the bottom rows of the switching matrix giving the \emph{Hermite Normal Form} ($\HNF$, for short) of the transpose matrix $V^T$. It is the generalization to higher rank of \cite[Prop.~5]{RT-WPS}. \\

\section{Preliminaries and notation}\label{sez:preliminari}

\subsection{Toric varieties}

 A \emph{$n$--dimensional toric variety} is an algebraic normal variety $X$ containing the \emph{torus} $T:=(\C^*)^n$ as a Zariski open subset such that the natural multiplicative self--action of the torus can be extended to an action $T\times X\rightarrow X$.

Let us quickly recall the classical approach to toric varieties by means of \emph{cones} and \emph{fans}. For proofs and details the interested reader is referred to the extensive treatments \cite{Danilov}, \cite{Fulton}, \cite{Oda} and the recent and quite comprehensive \cite{CLS}.

\noindent As usual $M$ denotes the \emph{group of characters} $\chi : T \to \C^*$ of $T$ and $N$ the \emph{group of 1--parameter subgroups} $\lambda : \C^* \to T$. It follows that $M$ and $N$ are $n$--dimensional dual lattices via the pairing
\begin{equation*}
\begin{array}{ccc}
M\times N & \longrightarrow & \Hom(\C^*,\C^*)\cong\C^*\\
 \left( \chi,\lambda \right) & \longmapsto
& \chi\circ\lambda
\end{array}
\end{equation*}
which translates into the standard paring $\langle u,v\rangle=\sum u_i v_i$ under the identifications $M\cong\Z^n\cong N$ obtained by setting $\chi(\t)=\t^{\u}:=\prod t_i^{u_i}$ and $\lambda(t)=t^{\v}:=(t^{v_1},\ldots,t^{v_n})$.

\subsubsection{Cones and affine toric varieties}

Define $N_{\R}:=N\otimes \R$ and $M_{\R}:=M\otimes\R\cong \Hom(N,\Z)\otimes\R \cong \Hom(N_{\R},\R)$.

\noindent A \emph{convex polyhedral cone} (or simply a \emph{cone}) $\sigma$ is the subset of $N_{\R}$ defined by
\begin{equation*}
    \sigma = \langle \v_1,\ldots,\v_s\rangle:=\{ r_1 \v_1 + \dots + r_s \v_s \in N_{\R} \mid r_i\in\R_{\geq 0} \}
\end{equation*}
The $s$ vectors $\v_1,\ldots,\v_s\in N_{\R}$ are said \emph{to generate} $\sigma$. A cone $\s=\langle \v_1,\ldots,\v_s\rangle$ is called \emph{rational} if $\v_1,\ldots,\v_s\in N$, \emph{simplicial} if $\v_1,\ldots,\v_s$ are $\R$--linear independent and \emph{non-singular} if $\v_1,\ldots,\v_s$ can be extended by $n-s$ further elements of $N$ to give a basis of the lattice $N$. A 1--generated cone $\rho=\langle\v\rangle$ is also called a \emph{ray}.

\noindent A cone $\s$ is called \emph{strictly convex} if it does not contain a linear subspace of positive dimension of $N_{\R}$. The linear span of a cone $\s$ will be denoted by $\mathcal{L}(\s)$ and by definition $\dim(\s):=\dim(\mathcal{L}(\s))$.  For a $n$--generated cone $\s=\langle \v_1,\ldots,\v_n\rangle$ we will set
$$ \det(\s):=|\det(\n_1,\ldots,\n_n)|$$
where $\n_i$ is the generator of the monoid $\langle\v_i\rangle\cap N$. A $n$--generated cone $\s$ will be called \emph{unimodular} when $\det(\s)=1$. Then a $n$--generated cone is simplicial if and only if its determinant is non zero and it is non-singular if and only if it is unimodular.

\noindent The \emph{dual cone $\s^{\vee}$ of $\s$} is the subset of $M_{\R}$ defined by
\begin{equation*}
    \sigma^{\vee} = \{ \u \in M_{\R} \mid \forall\ \v \in \sigma \quad \langle \u, \v \rangle \ge 0 \}
\end{equation*}
A \emph{face $\tau$ of $\s$} (denoted by $\tau <\s$) is the subset defined by
\begin{equation*}
    \tau = \sigma \cap \u^{\bot} = \{\v \in \sigma \mid \langle \u, \v \rangle = 0 \}
\end{equation*}
for some $\u\in \sigma ^{\vee}$. Observe that also $\tau$ is a cone.

\noindent Gordon's Lemma (see \cite{Fulton} \S~1.2, Proposition 1) ensures that the semigroup $S_{\s}:=\s^{\vee}\cap M$ is \emph{finitely generated}. Then also the associated $\C$--algebra $A_{\s}:=\C[S_{\s}]$ is finitely generated. A choice of $r$ generators gives a presentation of $A_{\s}$
\begin{equation*}
    A_{\s}\cong \C[X_1,\dots,X_r]/I_{\s}
\end{equation*}
where $I_{\s}$ is the ideal generated by the relations between generators. Then
\begin{equation*}
    U_{\s}:=\mathcal{V}(I_{\s})\subset\C^r
\end{equation*}
turns out to be an \emph{affine toric variety}. In other terms an affine toric variety is given by $U_{\s}:=\Spec(A_{\s})$. Since a closed point $x\in U_{\s}$ is an evaluation of elements in $\C[S_{\s}]$ satisfying the relations generating $I_{\s}$, then it can be identified with a semigroup morphism $x:S_{\s}\rightarrow\C$ assigned by thinking of $\C$ as a multiplicative semigroup. In particular the \emph{characteristic morphism}
\begin{equation}\label{caratteristico}
\begin{array}{cccc}
x_{\s}&:\s^{\vee}\cap M & \longrightarrow & \C\\
      &      \u & \longmapsto & \left\{\begin{array}{cc}
                                         1 & \text{if $\u\in\s^{\bot}$} \\
                                         0 & \text{otherwise}
                                       \end{array}
      \right.
\end{array}
\end{equation}
which is well defined since $\s^{\bot}<\s^{\vee}$, defines a \emph{characteristic point} $x_{\s}\in U_{\s}$ whose torus orbit $O_{\s}$ turns out to be a $(n-\dim(\s))$--dimensional torus embedded in $U_{\s}$ (see e.g. \cite{Fulton} \S~3).

\subsubsection{Fans and toric varieties}

A \emph{fan} $\Si$ is a finite set of cones $\s\subset N_{\R}$ such that
\begin{enumerate}
  \item for any cone $\s\in\Si$ and for any face $\tau<\s$ then $\tau\in\Si$,
  \item for any $\s,\tau\in\Si$ then $\s\cap\tau<\s$ and $\s\cap\tau<\tau$.
\end{enumerate}
For every $i$ with $0\leq i\leq n$ denote by $\Si(i)\subset \Si$ the subset of $i$--dimensional cones, called the \emph{$i$--skeleton of $\Si$}. A fan $\Si$ is called \emph{simplicial} if every cone $\s\in\Si$ is simplicial and \emph{non-singular} if every such cone is non-singular. The \emph{support} of a fan $\Si$ is the subset $|\Si|\subset N_{\R}$ obtained as the union of all of its cones i.e.
\begin{equation*}
    |\Si|:= \bigcup_{\s\in\Si} \s \subset N_{\R}\ .
\end{equation*}
If $|\Si|=N_{\R}$ then $\Si$ will be called \emph{complete} or \emph{compact}.

Since for any face $\tau <\s$ the semigroup $S_{\s}$ turns out to be a sub-semigroup of $S_{\tau}$, there is an induced immersion $U_{\tau}\hookrightarrow U_{\s}$ between the associated affine toric varieties which embeds $U_{\tau}$ as a principal open subset of $U_{\s}$. Given a fan $\Si$ one can construct \emph{an associated toric variety $X(\Si)$} by patching all the affine toric varieties $\{U_{\s}\ |\ \s\in\Si \}$ along the principal open subsets associated with any common face. Moreover \emph{for every toric variety $X$ there exists a fan $\Si$ such that $X\cong X(\Si)$} (see \cite{Oda} Theorem 1.5). It turns out that (\cite{Oda} Theorems 1.10 and 1.11; \cite{Fulton} \S~2):
\begin{itemize}
  \item \emph{$X(\Si)$ is non-singular if and only if the fan $\Si$ is non-singular,}
  \item \emph{$X(\Si)$ is complete if and only if the fan $\Si$ is complete.}
\end{itemize}

 Let us now introduce some non-standard notation.

  \begin{definition}\label{def:1-gen-fan} A \emph{$1$--generated rational fan} (\emph{$1$--fan} for short) $\Si$ is a fan whose 1--skeleton is given by a set $\Sigma(1)=\{\langle\v_1\rangle,\ldots,\langle\v_s\rangle\}\subset N_{\R}$ of rational rays and whose cones $\s\subset N\otimes\R$ are generated by \emph{any} proper subset of $\Sigma(1)$. When it makes sense, we will write
 \begin{equation}\label{def:genfan}
     \Si=\fan(\mathbf{v}_1,\ldots,\mathbf{v}_s)\ .
 \end{equation}
 \end{definition}

 \begin{examples}\label{ex:1-gen-fan} \begin{enumerate}
                    \item Given $$\v_1=\left(
                                         \begin{array}{c}
                                           1 \\
                                           0 \\
                                         \end{array}
                                       \right), \v_2=\left(
                                         \begin{array}{c}
                                           0 \\
                                           1 \\
                                         \end{array}
                                       \right),\v_3=\left(
                                         \begin{array}{c}
                                           -1 \\
                                           -1 \\
                                         \end{array}
                                       \right)\in N=\Z^2
                    $$then $\fan(\v_1,\v_2,\v_3)$ is well defined and gives a fan of $\P^2$.
                    \item Given $$\v_1=\left(
                                         \begin{array}{c}
                                           1 \\
                                           0 \\
                                         \end{array}
                                       \right), \v_2=\left(
                                         \begin{array}{c}
                                           0 \\
                                           1 \\
                                         \end{array}
                                       \right),\v_3=\left(
                                         \begin{array}{c}
                                           1 \\
                                           1 \\
                                         \end{array}
                                       \right)\in N=\Z^2
                    $$then $\fan(\v_1,\v_2,\v_3)$ is not well defined.
                  \end{enumerate}
 \end{examples}Let us now introduce some non-standard notation.

\begin{definition}\label{def:1-simp-gen-fan} A \emph{rational simplicial fan} $\Si$ is a fan whose 1--skeleton is given by a set $\Sigma(1)=\{\langle\v_1\rangle,\ldots,\langle\v_s\rangle\}\subset N_{\R}$ of rational rays and whose cones $\s\subset N\otimes\R$ are generated by \emph{a suitable choice of} proper subsets of $\Sigma(1)$, such that all the chosen subsets generate simplicial cones. In general, for a given set of generators these fans are not unique and $\mathcal{SF}(\v_1,\ldots,\v_s)$ will denote the set of rational simplicial fans whose 1--skeleton is given by $\Sigma(1)=\{\langle\v_1\rangle,\ldots,\langle\v_s\rangle\}\subset N_{\R}$ and whose support is $|\Si|=\langle\v_1,\ldots,\v_s\rangle$. The matrix $V=(\v_1,\ldots,\v_s)$ will be called \emph{a fan matrix} for every fan in $\SF(V):=\mathcal{SF}(\v_1,\ldots,\v_s)$: it is determined up to permutations of the $s$ generators. A fan matrix $V$ is called \emph{reduced} if any column $\v_i$ of $V$ is the generator of the monoid $\langle\v_i\rangle\cap N$.
\end{definition}

\begin{examples} \label{ex:SF}\begin{enumerate}
                   \item Consider $\{\v_1,\v_2,\v_3\}\subset N=\Z^2$ like in Example \ref{ex:1-gen-fan} (1). Then $\mathcal{SF}(\v_1,\v_2,\v_3)=\{\fan(\v_1,\v_2,\v_3)\}$.
                   \item Consider $\{\v_1,\v_2,\v_3\}\subset N=\Z^2$ like in Example \ref{ex:1-gen-fan} (2). Then $$\mathcal{SF}(\v_1,\v_2,\v_3)=\{\Si_1=\{\langle\v_1,\v_3\rangle,\langle\v_2,\v_3\rangle, \langle\v_1\rangle,\langle\v_2\rangle,\langle\v_3\rangle,\langle 0 \rangle\}\}\ .$$
                   \item Given $$\v_1=\left(
                                        \begin{array}{c}
                                          1 \\
                                          0 \\
                                          0 \\
                                        \end{array}
                                      \right), \v_2=\left(
                                        \begin{array}{c}
                                          0 \\
                                          1 \\
                                          0 \\
                                        \end{array}
                                      \right),\v_3=\left(
                                        \begin{array}{c}
                                          1 \\
                                          0 \\
                                          1 \\
                                        \end{array}
                                      \right), \v_4=\left(
                                        \begin{array}{c}
                                          0 \\
                                          1 \\
                                          1 \\
                                        \end{array}
                                      \right)\in N=\Z^3
                   \ .$$Then $\mathcal{SF}(\v_1,\v_2,\v_3,\v_4)=\{\Si_1,\Si_2\}$ with
                   \begin{eqnarray*}
                     \Si_1 &=& \{\langle\v_1,\v_2,\v_3\rangle,\langle\v_2,\v_3,\v_4\rangle,\langle\v_1,\v_2\rangle,
                   \langle\v_1,\v_3\rangle,\langle\v_2,\v_3\rangle,\langle\v_2,\v_4\rangle,\langle\v_3,\v_4\rangle,\\
                   &&\langle\v_1\rangle,\langle\v_2\rangle,\langle\v_3\rangle,\langle\v_4\rangle,\langle 0\rangle\} \\
                     \Si_2 &=&  \{\langle\v_1,\v_2,\v_4\rangle,\langle\v_1,\v_3,\v_4\rangle,\langle\v_1,\v_2\rangle,
                   \langle\v_1,\v_3\rangle,\langle\v_1,\v_4\rangle,\langle\v_2,\v_4\rangle,\langle\v_3,\v_4\rangle,\\
                   &&\langle\v_1\rangle,\langle\v_2\rangle,\langle\v_3\rangle,\langle\v_4\rangle,\langle 0\rangle\}
                   \end{eqnarray*}
                   giving two simplicial fans obtained by putting a diagonal facet inside the non-simplicial cone $\langle\v_1,\v_2,\v_3,\v_4\rangle$.
                 \end{enumerate}
\end{examples}

\begin{definition}\label{def:dd} If the set of generators $\{\v_1,\ldots,\v_s\}\subset N$ is such that the cardinality $|\mathcal{SF}(\v_1,\ldots,\v_s)|=1$, then the unique element $\Si\in\mathcal{SF}(\v_1,\ldots,\v_s)$ is called a \emph{$1$--detected fan} and the associated normal toric variety $X(\Si)$ is said a \emph{divisorially detected variety} (or simply \emph{dd-variety}).
\end{definition}

\begin{remarks}\label{rems:SF} Some immediate observations are the following.
\begin{enumerate}
  \item If $s\leq n=\rk N$ and $\v_1,\ldots,\v_s$ are $\R$-linearly independent then $$|\mathcal{SF}(\v_1,\ldots,\v_s)|=1$$
  since its unique $1$--detected fan is given by the $1$--fan $\fan(\v_1,\ldots,\v_s)$.
  \item Clearly a $1$--fan is always a $1$--detected fan, but the converse is false: recall Example \ref{ex:SF} (2).
  \item If $n=2$ and the fan matrix $V=(\v_1,\ldots,\v_s)$ has maximum rank 2, then $$|\mathcal{SF}(\v_1,\ldots,\v_s)|=1 \hbox{ for any } s\geq 2.$$ In fact for $s=2$ we are in the previous case; for $s\geq 3$, construct the 1--detected fan $\Si\in\mathcal{SF}(\v_1,\ldots,\v_s)$ by considering all the rank 2 sub-matrices $(\v_i,\v_j)$ of $V$ such that every $\v_k$, with $k\neq i,j$, is not in the interior of the cone $\s_{ij}:=\langle\v_i,\v_j\rangle$. Assume all the $\s_{ij}$ to be the maximal cones of $\Si$ and the remaining cones of $\Si$ to be given by all their faces.
  \item As a consequence of the previous fact, we get the non surprising assertion that \emph{every 2-dimensional normal and $\Q$--factorial toric variety is a dd variety}.
\end{enumerate}
\end{remarks}

\subsection{Normal forms for matrices}

Let us firstly recall the following:
\begin{definition}
A matrix is in {\em row echelon form}  if
\begin{itemize}
\item All nonzero rows  are above any rows of all zeroes (all zero rows, if any, lying on the bottom of the matrix);
\item the first nonzero entry from the left of a nonzero row is always strictly on the right of the first nonzero entry of the previous row.
\end{itemize}
\end{definition}
\subsubsection{Hermite normal form}\label{HNF}
Hermite normal form is a particular case of row e\-che\-lon form.\\
It is well known that Hermite algorithm provides an effective way to determine a basis of a subgroup of $\mathbb{Z}^n$. We briefly recall the definition and the main properties.
For details, see for example \cite{Cohen}.

\begin{definition}\label{def:HNF}
An $m\times n$ matrix $M=(m_{ij})$ with integral coefficients is in \emph{Hermite normal form} (abbreviated $\HNF$) if there exists $r\leq m$ and a strictly increasing map $f:\{1,\ldots,r\} \to \{1,\ldots,n\}$ satisfying the following properties:
\begin{enumerate}
\item For $1\leq i\leq r$, $m_{i,f(i)}\geq 1$, $m_{ij}=0$ if $j<f(i)$ and $0\leq m_{i,f(k)}< m_{k,f(k)}$ if $i<k$.
\item The last $m-r$ rows of $M$ are equal to $0$.
\end{enumerate}
\end{definition}

\begin{theorem}[\cite{Cohen} Theorem 2.4.3]\label{thm:cohen}
Let $A$ be an $m\times n$ matrix with coefficients in $\mathbb{Z}$. Then there exists a unique $m\times n$ matrix $B=(b_{ij})$ in $\HNF$ of the form $B=U\cdot A$ where $U\in\GL_m(\mathbb{Z})$.
\end{theorem}

We will refer to matrix $B$ as the $\HNF$ of matrix $A$. The construction of $B$ and $U$ is effective, see \cite[Algorithm 2.4.4]{Cohen}, based on Eulid's algorithm for greatest common divisor.

\begin{proposition}[\cite{Cohen},\S~2.4.3]\label{prop:HNFuno}
\begin{enumerate}
\item Let $\mathcal{L}$ be a subgroup of $\mathbb{Z}^n$, $V=\{{\v}_1,\ldots,{\v}_m\}$ a set of generators, and let $A$ be the $m\times n$ matrix having ${\v}_1,\ldots,{\v}_m$ as rows. Let $B$ be the $\HNF$ of $A$. Then the non zero rows of $B$ are a basis of $\mathcal{L}$.
\item Let $A$ be a $m\times n$ matrix, and let $B=U\cdot A^T$ be the $\HNF$ of the transposed of $A$, and let $r$ such that the first $r$ rows of $B$ are non zero. Then a $\mathbb{Z}$-basis for the kernel of $A$ is given by the last $m-r$ rows of $U$.
\end{enumerate}
\end{proposition}

\begin{remark}\label{rem:HNFestesa}
The Hermite Normal Form can be defined also for $m \times n$ matrices with rational coefficients. Given such a matrix, one simply multiplies it by a common multiple $D$ of denominators of the entries, compute the $\HNF$ and then divides by $D$.\\
\end{remark}
\subsubsection{Dual discrete subgroups in $\Q^n$}\hfill
\begin{definition}
Given a discrete subgroup $\mathcal{L}$ of $\Q^n$, let $\mathcal{L}_{\Q}$ be the $\Q$--span of $\mathcal{L}$ in $\Q^n$. The \emph{dual subgroup} of $\mathcal{L}$ is the discrete subgroup defined by
$${\mathcal{L}}^*=\{\mathbf{x}\in \mathcal{L}_{\Q}\ |\ \mathbf{x}\cdot\mathbf{y}\in\Z, \forall \mathbf{y}\in \mathcal{L}\}\subseteq \Q^n$$
\end{definition}
\begin{definition} If $A$ is a $m\times n$ matrix of rank $m$,  we will call the matrix
$$A^*:=(A\cdot A^T)^{-1}\cdot A$$
the \emph{transverse} of $A$; it is also called the \emph{contragredient matrix} by some authors. Notice that $A\cdot A^T$ is a non-singular square matrix if $A$ is a $m\times n$ matrix of rank $m$; therefore $A^*$ is well-defined.

\noindent In particular, when $A$ is a non-singular square matrix,  $A^*=(A^T)^{-1}$ (see \cite[\S~1.3]{RT-WPS}).
\end{definition}
The following facts are well-known:
\begin{proposition}\label{prop:duallattice}
Let $\mathcal{L}, \mathcal{L}_1,\mathcal{L}_2$ be discrete subgroup of $\Q^n$.
\begin{enumerate}
\item Let $m$ be the rank of $\mathcal{L}$, and $B$ be the $m\times n$ matrix having on the rows a basis of $\mathcal{L}$. Then a $\Z$-basis of ${\mathcal{L}}^*$ is given by the rows of the transverse matrix $B^*$ of $B$.
\item $\mathcal{L}^{**}=\mathcal{L}$.
\item $(\mathcal{L}_1\cap \mathcal{L}_2)^*$ is generated by $\mathcal{L}_1^*$ and $\mathcal{L}_2^*$.
\item $\mathcal{L}_1\cap \mathcal{L}_2$ is the dual of the group generated by $\mathcal{L}_1^*$ and  $\mathcal{L}_2^*$.
\end{enumerate}
\end{proposition}
\begin{proof} (1) Notice first of all that the spaces spanned over $\Q$ by the rows of $B^*$ and $B$ are the same.  We have
$B^*\cdot B^T=\mathbf{I}_m$, so that the rows of $B^*$ belong to ${\mathcal{L}}^*$. Conversely, let $\mathbf{x} $ be a (row)  vector in ${\mathcal{L}}^*$; then $\mathbf{x}\cdot B^T=\mathbf{y}\in\Z^m$ and since $\mathbf{x}\in \mathcal{L}_{\Q}$ there exists $\mathbf{w}\in \Q^m$ such that $\mathbf{w}\cdot B=\mathbf{x}$. Then $\mathbf{y}=\mathbf{w}\cdot(B\cdot B^T)$, so that
$$\mathbf{x}=\mathbf{w}\cdot B=\mathbf{y}\cdot(B\cdot B^T)^{-1}\cdot B=\mathbf{y}\cdot B^*\in\mathcal{L}_r(B^*).$$
(2) Let $B, B^*$ as above; then it is easily seen that  $(B^*\cdot (B^*)^T)^{-1}\cdot B^*=B$, proving the claim.\\
(3) is clear from the definition of duality.\\
(4) follows from (2) and (3). \end{proof}

\subsubsection{Intersection of subgroups of $\Z^n$}\label{sssez:intersezione}
Let $\mathcal{L}_1$ and $\mathcal{L}_2$ be subgroups of $\Z^n$. The previous Proposition~\ref{prop:duallattice} provides a constructive method to compute a basis of the intersection $\mathcal{L}_1\cap\mathcal{L}_2$:
\begin{itemize}
\item compute matrices $B_1$ and $B_2$  having on the rows bases of  $\mathcal{L}_1$ and $\mathcal{L}_2$ respectively;
\item for $i=1,2$ compute $B^*_i$; define
$$M=\begin{pmatrix} B^*_1\\
B^*_2
\end{pmatrix}
$$
\item  compute $A'=\HNF(M)$ and extract from $A'$ the submatrix $A$ of all non-zero rows of $A'$.
\item compute $B=A^*$; then the rows of $B$ form a basis for $\mathcal{L}_1\cap\mathcal{L}_2$.
\end{itemize}
Of course, the procedure can be adapted in order to calculate the intersection of more than two subgroups.
\subsubsection{Smith Normal Form ($\SNF$)}\par

We recall the well-known Elementary Divisor Theorem, which will be often used in the sequel. For details see for example \cite[\S~2.4]{Cohen}
\begin{theorem}[Elementary Divisor Theorem]\label{teo:EDT}
Let $\mathcal{L}$ be a subgroup of $\Z^m$. There exist a basis $\mathbf{f}_1,\ldots,\mathbf{f}_m$ of $\Z^m$, a positive integer $k\leq m$ (called the \emph{rank} of $\mathcal{L}$) and integers $c_1,\ldots,c_k>0$ such that $c_i$ divides $c_{i+1}$ for $i=1,\ldots,k-1$, and $c_1\mathbf{f}_1,\ldots,c_k\mathbf{f}_k$ is a basis of $\mathcal{L}$.
\end{theorem}

An effective version of Theorem \ref{teo:EDT} is given by the Smith Normal Form algorithm:

\begin{definition}
\label{def:SNF} Let $A=(a_{ij})$ be a $d\times m$ matrix with integer entries. Then $A$ is said in \emph{Smith Normal Form} (abbreviated $\SNF$) if there exist a positive integer $k\leq d$  and integers $c_1,\ldots,c_k>0$ such that $c_i$ divides $c_{i+1}$ for $i=1,\ldots,k-1$ such that  $a_{ii}=c_i$ if $1\leq i\leq k$ and $a_{ij}=0$ if either $i\not=j$ or $i>k$.
\end{definition}

\begin{theorem}[\cite{Cohen}, Algorithm 2.4.14]\label{teo:SNF} Let $A=(a_{ij})$ be a $d\times m$ matrix with integer entries. Then it is possible to compute matrices $\alpha\in\GL_d(\Z),\beta\in\GL_m(\Z)$ such that $\alpha\cdot A\cdot \beta$ is in $SNF$.
\end{theorem}

Let $\mathcal{L}$ be a subgroup of $\Z^m$, and suppose to know generators $\mathbf{g}_1,\ldots,\mathbf{g}_d$ of $\mathcal{L}$. Let $A$ be the matrix having   $\mathbf{g}_1,\ldots,\mathbf{g}_d$ as rows, and let $\alpha\in\GL_d(\Z),\beta\in\GL_m(\Z)$ such that $\alpha A\beta$ is in $\SNF$, with diagonal coefficients $c_i, i=1,\ldots,k$, as in Theorem \ref{teo:SNF}. Then the rows of $\beta^{-1}$ provide a basis $\mathbf{f}_1,\ldots,\mathbf{f}_m$ of $\Z_m$ such that $c_1\mathbf{f}_1,\ldots,c_k\mathbf{f}_k$ is a basis of $\mathcal{L}$, according to Theorem \ref{teo:EDT}.\\

Let us fix some notation. If $A$ is a $d\times m$ matrix with integer coefficients, ${\mathcal L}_r(A)$ (${\mathcal L}_c(A)$ respectively) denotes the lattice spanned by the rows (resp. columns) of $A$.\\
A matrix, and in particular a vector, is said \emph{positive} (resp. \emph{strictly positive}) if each entry is $\geq 0$ (resp. $>0$).\\
We write $A_1\sim A_2$ if $A_1$ and $A_2$ are in the same orbit by left multiplication by $\GL_d(\Z)$.\\

\section{Poly weighted spaces: motivations and main geometrical results}\label{sez:motivazioni}
\subsection{Divisors, fan matrices and weight matrices}

Consider the normal toric variety $X=X(\Si)$.
If $\mathcal{W}(X)$ denotes the group of Weil divisors of $X$ then its subgroup of \emph{torus--invariant Weil divisors} is given by
\begin{equation*}
    \mathcal{W}_T(X)=\left\langle D_{\rho } \mid \rho \in \Sigma (1)\right\rangle_{\Z} =
    \bigoplus_{\rho \in \Sigma (1)}\Z\cdot D_{\rho }
\end{equation*}
where $D_{\rho}=\overline{O}_{\rho}$ is the closure of the torus orbit of the ray $\rho$. Let $\mathcal{P}(X)\subset\mathcal{W}(X)$ be the subgroup of \emph{principal divisors} and $\v_{\rho}$ be the minimal integral vector in $\rho$ i.e. the generator of the monoid $\rho\cap N$. Then the morphism
\begin{equation}\label{div}
\begin{array}{llll}
div : & M & \longrightarrow & \mathcal{P}(X)\cap \mathcal{W}_{T}(X)=:
\mathcal{P}_{T}(X) \\
& \u & \longmapsto & div(\u):=\sum_{\rho \in \Sigma (1)}\langle \u,\v_{\rho }\rangle
D_{\rho }
\end{array}
\end{equation}
is surjective. Let $\Pic(X)$ be the group of line bundles modulo isomorphism. It is well known that for an \emph{irreducible} variety $X$ the map $D\mapsto\mathcal{O}_X(D)$ induces an isomorphism $\mathcal{C}(X)/\mathcal{P}(X)\cong\Pic(X)$, where $\mathcal{C}(X)\subset\mathcal{W}(X)$ denotes the subgroup of Cartier divisors. The divisor class group is defined as the group of Weil divisors modulo rational (hence linear) equivalence, i.e. $\Cl(X):=\mathcal{W}(X)/\mathcal{P}(X)$. Then the inclusion $\mathcal{C}(X)\subset\mathcal{W}(X)$ passes through the quotient giving an immersion $\Pic(X)\hookrightarrow \Cl(X)$. One of main results on divisors of toric varieties is then the following

\begin{theorem}[\cite{Fulton} \S~3.4, \cite{CLS} Proposition 4.2.5]\label{thm:divisors} For a toric variety $X=X(\Sigma )$ the following sequence is exact
\begin{equation}
\def\objectstyle{\displaystyle}
\xymatrix@1{M \ar[r]^-{div} & *!U(.45){\bigoplus\limits_{\rho \in \Sigma (1)} \Z \cdot D_{\rho}}
\ar[r]^-d & \Cl (X) \ar[r] & 0 }
\label{deg sequence}
\end{equation}
Moreover if $\Sigma (1)$ generates $N_{\R}$ then the morphism $div$ is injective giving the following exact sequences
\begin{equation}
\def\objectstyle{\displaystyle}
\xymatrix{
& 0 \ar[d] & 0 \ar[d] & 0 \ar[d] & \\
0 \ar[r] & M \ar[r]^{div}\ar[d]^{\parallel} &
\mathcal{C}_T (X) \ar[r]\ar[d] & {\Pic(X)} \ar[r]\ar[d] & 0 \\
0 \ar[r] & M \ar[r]^-{div}\ar[d] & *!U(.40){\bigoplus\limits_{\rho \in \Sigma (1)} \Z \cdot D_{\rho}}
\ar[r]^-d & \Cl (X) \ar[r] & 0 \\
 & 0 &  &  & }
\label{div-diagram}
\end{equation}
where $\mathcal{C}_T (X)=\mathcal{C} (X)\cap \mathcal{W}_{T}(X)$. In particular $\Pic(X)$ and $\Cl(X)$ turn out to be completely described by means of torus invariant divisors and
\[
\rk \left( \Pic \left( X\right) \right) \leq \rk \left( \Cl\left( X\right)
\right) =\left| \Sigma (1)\right| -n\ .
\]
Moreover if $\Si$ contains a $n$-dimensional cone then the first sequence in $(\ref{div-diagram})$ splits implying that $\Pic(X)$ is a free abelian group.
\end{theorem}
In the following $X=X(\Si)$ will be a $\Q$--factorial and complete $n$--dimensional toric variety, meaning that $\Si$ is simplicial and its support $|\Si|=N_{\R}$, respectively. In particular completeness of $X$ implies that it cannot admit torus factors, or equivalently that $\Si(1)$ generates the whole $N_{\R}$. Then diagram  (\ref{div-diagram}) in Theorem \ref{thm:divisors} holds and $|\Si(1)|=n+r$, where $r:=\rk \Pic(X)\geq 1$ is the \emph{Picard number of} $X$, simply called \emph{the rank of $X$}. Recalling Definitions \ref{def:1-gen-fan} and \ref{def:1-simp-gen-fan}, let $V=(\v_1,\ldots,\v_{n+r})$ be the $n\times (n+r)$ \emph{reduced fan matrix of} $\Si$ i.e. $\v_i$ is the unique generator of the monoid $\langle \v_i \rangle\cap N$.\\ Observe that the transposed matrix $V^T$ is a representative matrix of the $\Z$-linear morphism $div$ defined in (\ref{div}) and appearing in diagram (\ref{div-diagram}).

\begin{definition}[\emph{(Reduced) Weight matrix} of a $\Q$--factorial complete toric variety]\label{def:Q} Let $X$ be a $\Q$--factorial complete $n$--dimensional toric variety of rank $r$ and consider the exact sequence given by the second row of the diagram (\ref{div-diagram}) in Theorem \ref{thm:divisors}, which is
\begin{equation}\label{div-sequence}
    \xymatrix{0 \ar[r] & M\ar[r]^-{div}& \mathcal{W}_T(X)=\Z^{n+r}
\ar[r]^-{d} & \Cl(X) \ar[r]& 0}\ .
\end{equation}
Dualizing this sequence, one gets the following exact sequence of free abelian groups
\begin{equation}\label{HomZ-div-sequence}
                    \xymatrix{0 \ar[r] & \Hom(\Cl(X),\Z) \ar[r]^-{d^{\vee}} & \Hom(\mathcal{W}_T(X),\Z)= \Z^{n+r}
 \ar[r]^-{div^{\vee}} & N }
\end{equation}
The transposed matrix $Q$ of a $(n+r)\times r$ representative matrix $Q^T$  of the $\Z$-linear morphism $d^{\vee}$ is called a \emph{weight matrix of $X$}.

A weight matrix $Q$ is called \emph{reduced} if for every $i=1,\ldots,n+r$, $\mathcal{L}_r(Q^i)$ has no cotorsion in $\Z^{n+r-1}$, where $Q^i$ denotes the  matrix obtained by removing from $Q$ the $i$-th column and $\mathcal{L}_r(Q^i)$ the lattice spanned by the rows of $Q^i$, as defined at the end of \S\ref{sez:preliminari}.
\end{definition}

\begin{remark}\label{rem:} Tensoring the short exact sequence (\ref{div-sequence}) by the field $\Q$, the torsion subgroup $\Tors(\Cl(X))$ of the class group $\Cl(X)$ is killed and one can think of a weight matrix $Q$ as a representative matrix of the $\Q$-linear morphism $d$. This is clearly no more true by the $\Z$-linear point of view, so giving the motivation for defining a weight matrix as the transposed of a representative matrix of the dual $\Z$-linear morphism $d^{\vee}$.

As observed \cite[(9.5.8)]{CLS}, recalling the definition of Gale duality originally given by Oda and Park in \cite{Oda-Park}, a (not necessarily reduced) weight matrix $Q$ is a Gale dual matrix of the transposed $V^T$ of a reduced fan matrix $V$. By abuse of notation we will say that $V$ and $Q$ are \emph{Gale dual matrices}. In \S\,\ref{sez:GaleDuality} we will give a linear algebraic characterization of fan matrices, weight matrices and Gale duality, in terms of $F$--matrices, $W$--matrices and their reduction (see Definitions \ref{def:Fmatrix}, \ref{def:Wmatrix}, \ref{def:F-red} and \ref{def:W-red}).

Let us finally notice that the reduction of a weight matrix as defined in the last part of Def.~\ref{def:Q} is actually the generalization, on the rank $r$, of the standard definition of a reduced weight system of a weighted projective space, as observed in Remark~\ref{rem:WPS}. Nevertheless it is cumbersome and we think that the more natural definition of a reduced weight matrix is that given in Def.~\ref{def:W-red}. Their equivalence is proved in the Reduction Theorem~\ref{thm:riduzione}.
\end{remark}

One of the main geometric results of this paper is the following:

\begin{theorem}\label{thm:pi11=Tors}  Let $X=X(\Si)$ be a $\Q$--factorial complete toric variety. For any ray $\rho\in\Si(1)$ let $\v_{\rho}$ be the minimal integral vector of $\rho$ and consider the $\Z$-module
\begin{equation}\label{NSigma}
    N_{\Si(1)}:=\langle\v_{\rho}\ |\ \rho\in\Si(1)\rangle_{\Z}
\end{equation}
as a subgroup of the lattice $N$. Let $V$ be a fan matrix of $\Si$, and let $T_n$ be the upper $n\times n$ part of the $\HNF$ of $V^T$. Then
 \begin{equation}\label{TorsdiChow}
   N/N_{\Sigma(1)}\cong \Tors(\Cl(X))\cong \Z^n/ \mathcal{L}_r(T_n)\,,
\end{equation}
where $\mathcal{L}_r(T_n)$ is the lattice spanned by the rows of $T_n$ (as defined at the end of \S~\ref{sez:preliminari}).
\end{theorem}

A proof is given in \S\ref{ssez:dimostrazione di pi11=Tors}.

\begin{remark} The group (\ref{TorsdiChow}) turns out to be isomorphic to the fundamental group in codimension 1 $\pi_1^1(X)$, as defined in \cite{Buczynska}.
\end{remark}

\subsection{Poly weighted spaces (PWS)}\label{ssez:PWS}

A first immediate consequence of previous Proposition \ref{prop:HNFuno}  and Theorem \ref{thm:pi11=Tors} is that

\begin{proposition}\label{prop:PWS} Let $X=X(\Si)$ be a $\Q$--factorial complete $n$--dimensional toric variety of rank $r$. Let $V$ be a fan matrix of $\Sigma$. Then the following are equivalent:
\begin{enumerate}
  \item the group $N/N_{\Si(1)}$ is trivial,

   \item the $\HNF$ of the transposed matrix $V^T$ is given by $\left(
                                                                   \begin{array}{c}
                                                                     \mathbf{I}_n \\
                                                                     \mathbf{0}_{r,n} \\
                                                                   \end{array}
                                                                 \right)
    $,
  \item the lattice $\mathcal{L}_c(V)$ spanned by the columns of $V$ coincides with the whole $\Z^n$,

   \item the  fan matrix $V$ has coprime $n\times n$ minors.
\end{enumerate}
\end{proposition}
A proof is given in \ref{ssez:dimostrazione di PWS}.

\begin{definition}\label{def:PWS} A \emph{poly weighted space} (PWS) is a $\Q$--factorial complete toric variety satisfying the equivalent conditions of Proposition \ref{prop:PWS}.
\end{definition}

\begin{remarks}\label{rem:motivazioni} Let us list an amount of motivating reasons supporting the previous definition.
\begin{itemize}
  \item[1.] Let $X$ be a $n$--dimensional PWS of rank 1. Then $X$ is a weighted projective space (WPS) of weights given by the entries of the $1\times(n+1)$ weight matrix $Q=(q_1,\ldots,q_{n+1})$. In particular, the previous Remark \ref{rem:} clarifies that $q_1,\ldots,q_{n+1}$ give a reduced weight system of the WPS $X=\P(q_1,\ldots,q_{n+1})$ if and only if $Q$ is a reduced weight matrix.
  \item[2.] Let $Q=(\q_1,\ldots,\q_{n+r})$ be a $r\times(n+r)$ reduced weight matrix of a PWS $X=X(\Si)$. The Cox presentation of $X$ as a geometric quotient \cite{Cox} shows that
      $$
      X\cong \left(\C^{n+r}\setminus Z_{\Si}\right)/\left(\C^*\right)^r
$$
where the action of $(\C^*)^r$ is completely described by the weight matrix  $Q$ as follows
\begin{equation*}
\quad\quad\quad\xymatrix{(\t,\z)\in (\C^*)^r\times\C^{n+r}\ar@{|->}[r]&\left(\prod_{i=1}^r t_i^{q_{i,1}}z_1,\ldots,\prod_{i=1}^r t_{i}^{q_{i,n+r}}z_{n+r}\right)\in\C^{n+r}}.
\end{equation*}
Notice that if $r=1$ this gives the usual quotient $\left(\C^{n+1}\setminus\{0\}\right)/\C^*$ by the weighted action
$(t,\z)\mapsto(t^{q_1}z_1,\ldots,t^{q_{n+1}}z_{n+1})$
defining $\P(q_1,\ldots,q_{n+1})$.
  \item[3.] Given a $d\times m $ matrix $A$, let us introduce the following notation: for every subset $I\subseteq\{1,\ldots,m\}$  let $A_I$ (resp. $A^I$)  be the submatrix of $A$ given by
    the columns indexed by $I$ (resp. indexed by the complementary subset $\{1,\ldots,m\}\backslash I$). Recalling \cite[Thm.~3]{RT-WPS}, a characterization of a fan matrix $V=(\v_1,\ldots,\v_{n+1})$ of $\P(q_1.\ldots,q_{n+1})$ is that:
  \begin{itemize}
    \item[i.] $\sum_{i=1}^{n+1}q_i\v_i=0$\,,
    \item[ii.] $\forall\,i=1,\ldots,n+1\quad\left|\det\left(V^{\{i\}}\right)\right|=q_i$\,.
  \end{itemize}
  The generalization of these conditions to the case $r>1$ is given by
  \begin{itemize}
    \item[I.] $Q\cdot V^T=0$, by the definition of Gale duality (see \S~\ref{ssez:Gale}),
    \item[II.] $\left|\det(V^I)\right | = \left|\det(Q_I)\right|$, for every  $I\subset\{1,\ldots,n+r\}$ such that $|I|=r$ (by Corollary \ref{cor:Gale-on-det} below).

  \end{itemize}
  \item[4.] Recalling \cite[Lemma~2.11]{BC} and \cite[Prop.~4.7]{Conrads}, a $\Q$--factorial complete toric variety $X$ of rank 1 is a WPS if and only if its divisor class group $\Cl(X)$ is torsion free, hence $\Cl(X)\cong\Z$ and $\Pic(X)$ is a free subgroup of $\Cl(X)$.

      \noindent Analogously, by Proposition \ref{prop:PWS}, a $\Q$--factorial complete toric variety $X$ of rank $r$ is a PWS if and only if its divisor class group $\Cl(X)$ is torsion free, hence $\Cl(X)\cong\Z^r$ and $\Pic(X)$ is a free subgroup, of maximum rank $r$, of $\Cl(X)$.
\end{itemize}
\end{remarks}

The last Remark \ref{rem:motivazioni}.4 deserves further study.
For a WPS $X=\P(q_1,\ldots,q_{n+1})$ with reduced weights,  a generator $L$ of $\Cl(X)\cong \Z$ can be found by taking any solution $(b_1,\ldots,b_{n+r})$ of the diophantine equation $\sum_{j=1}^{n+r}q_j x_j=1$ and setting $L:=\sum_{j=1}^{n+r}b_j D_j$, where $D_j$ is the torus invariant divisor giving the closure of the torus orbit of the ray $\langle\v_j\rangle$. Then a generator of $\Pic(X)$ is given by the line bundle $\mathcal{O}_X\left(\d L\right)\cong \mathcal{O}_X\left((\d/q_j)D_j\right)$, for every $1\leq j\leq n+1$, with $\d:=\lcm(q_1,\ldots,q_{n+1})$ \cite[Prop.~8]{RT-WPS}. In particular the inclusion $\Pic(X)\hookrightarrow \Cl(X)$ is described by the multiplication by $\d$.

\noindent In the general case of a PWS of rank $r$ these considerations admit the following ge\-ne\-ra\-li\-zation:

\begin{theorem}\label{thm:generazione} Let $X=X(\Si)$ be a $n$--dimensional PWS of rank $r$ admitting a reduced weight matrix $Q=(\q_1,\ldots,\q_{n+r})$.
\begin{enumerate}
\item Consider the matrix
$$U_Q=(u_{ij})\in\GL_{n+r}(\Z)\ :\ U_Q\cdot Q^T\ \text{ is a HNF matrix.}$$
Let us denote by $^rU_Q$ the submatrix of $U_Q$ given by the upper $r$ rows of $U_Q$. Then the rows of $^rU_Q$ describe the following set of generators of $\Cl(X)$
\begin{equation}\label{generazione}
    \forall\,1\leq i\leq r\quad L_i:=\sum_{j=1}^{n+r} u_{ij} D_j\in \mathcal{W}_T(X)\quad\text{and}\quad \Cl(X)=\bigoplus_{i=1}^r \Z[d(L_i)]
\end{equation}
where $d:\mathcal{W}_T(X)\to \Cl(X)$ is the morphism appearing in the exact sequence (\ref{div-sequence}).
\item
Let $V$ be a fan matrix of $X$. Define
  $$\mathcal{I}_\Si=\{I\subset\{1,\ldots,n+r\}:\left\langle V^I\right\rangle\in\Si(n)\}.$$
  Through the identification of $\Cl(X)$ with $\Z^r$, just fixed in (\ref{generazione}),
  $$\Pic(X)=\bigcap_{I\in\mathcal{I}_\Si} \mathcal{L}_c(Q_I).$$
  Therefore a basis of $\Pic(X)\subseteq \Cl(X)\simeq \Z^r$ can be computed by applying the procedure described in \S~\ref{sssez:intersezione}.
\item Let $\mathbf{b}_1,...,\mathbf{b}_r$ be a basis of $\Pic(X)$ in $\Z^r\simeq \Cl(X)$, and let $B$ be the $r\times r$ matrix  having $\mathbf{b}_1,...,\mathbf{b}_r$ on the rows.   Then a basis of $\mathcal{C}_T(X)\subseteq \Z^{n+r}\simeq \mathcal{W}_T(X)$ is given by the rows of the matrix
$$C=\begin{pmatrix}B & \mathbf{0}_{r,n}\\  \mathbf{0}_{n,r}& \mathbf{I}_{n}\end{pmatrix}\cdot U_Q= \begin{pmatrix}B\cdot\,^rU_Q\\  V\end{pmatrix}$$ where $V$ is the fan matrix of $X$ given by the lower $n$ rows of $U_Q$ (see Proposition \ref{prop:QdaVeviceversa}).
\item Setting $\d_{\Si}:=\lcm\left(\det(Q_I):I\in\mathcal{I}_\Sigma\right)$
then
$$\d_{\Si}\mathcal{W}_T(X)\subseteq \mathcal{C}_T(X)\ \text{and}$$
$$\d_{\Si}\ \text{divides the index}\ [\Cl(X):\Pic(X)]=[\mathcal{W}_T(X):\mathcal{C}_T(X)]\,.$$
\end{enumerate}
\end{theorem}

\begin{remark} Let us underline that the HNF of a matrix and the associated switching matrix are obtained by a well known algorithm, based on Euclid's algorithm for greatest common divisor (see e.g. \cite[Algorithm 2.4.4]{Cohen}). This algorithm is implemented in many computer algebra procedures. Due to the following Proposition \ref{prop:QdaVeviceversa} and part (1) of  Theorem \ref{thm:generazione} the switching matrix $U_Q\in\GL_{n+r}(\Z)$ turns out to encode both the generators of the divisor class group (given by the upper $r$ rows) and the fan matrix (given by the lower $n$ rows).\\
  Parts (2) and (3) of Theorem \ref{thm:generazione} provide effective methods to compute generators of the subgroup $\Pic(X)$ in $\Cl(X)$ and the subgroup $\mathcal{C}_T(X)$ of torus invariant Cartier divisor in $\mathcal{W}_T(X)$. Consequently, Theorem \ref{thm:generazione} allows us to determine the torsion group $\mathcal{T}:=\mathcal{W}_T(X)/\mathcal{C}_T(X)\cong \Cl(X)/\Pic(X)$ of a PWS.

  Moreover, Theorem \ref{thm:generazione} enables us to explicitly describe all the morphisms appearing in the commutative diagram (\ref{div-diagram}), by giving all the representative matrices over suitable fixed bases. Namely, let us fix
  \begin{itemize}
    \item the basis of $\mathcal{W}_T(X)$ given by the torus invariant divisors $\{D_1,\ldots,D_{n+r}\}$,
    \item the basis of $\Cl(X)$ given by $\{d(L_1),\ldots,d(L_r)\}$, as in (\ref{generazione}),
  \end{itemize}
  then the matrices $Q,V,C,B$ defined in Theorem \ref{thm:generazione} completely describe diagram (\ref{div-diagram}) as follows:
\begin{equation}
\def\objectstyle{\displaystyle}
\xymatrix{
& 0 \ar[d] && 0 \ar[d] && 0 \ar[d] & \\
0 \ar[r] & M \ar[rr]^-{\left(
                                         \begin{array}{c}
                                           \mathbf{0}_{n,r}\,|\,\mathbf{I}_n  \\
                                         \end{array}
                                       \right)}\ar[d]^{\parallel} &&
\mathcal{C}_T(X)\cong\Pic(X)\oplus M \ar[rr]^-{\left(
                                         \begin{array}{c}
                                           \mathbf{I}_r\,|\,\mathbf{0}_{r,n}  \\
                                         \end{array}
                                       \right)}\ar[d]^-{C^T} && {\Pic(X)} \ar[r]\ar[d]^-{B^T} & 0 \\
0 \ar[r] & M \ar[rr]^-{div}_-{V^T}\ar[d] && \mathcal{W}_T(X)=\bigoplus_{j=1}^{n+r} \Z \cdot D_{j}\ar[d]
\ar[rr]^-d_-Q && \Cl(X)\ar[d] \ar[r] & 0 \\
 & 0\ar[rr] && \mathcal{T}\ar[d]\ar[rr]^-= && \mathcal{T}\ar[r]\ar[d]&0\\
 & &&0&&0& }
\end{equation}

\end{remark}

\begin{example}\label{ex} Let us give an account of what observed until now by considering the following  matrix in row echelon form: $Q=\left(
      \begin{array}{cccc}
        1 & 1 & 0 & 0 \\
        0 & 1 & 1 & 2 \\
      \end{array}
    \right)\,.
$

\noindent One easily check that $Q$ is reduced. The HNF of $Q^T$ is given by
\begin{equation*}
    \left(
       \begin{array}{cc}
         1 & 0 \\
         0 & 1 \\
         0 & 0 \\
         0 & 0 \\
       \end{array}
     \right)= \left(
                \begin{array}{cccc}
                  1 & 0 & 0 & 0 \\
                  -1 & 1 & 0 & 0 \\
                  1 & -1 & 1 & 0 \\
                  0 & 0 & 2 & -1 \\
                \end{array}
              \right)\cdot Q^T= U_Q\cdot Q^T\,.
\end{equation*}
The matrix $V=\left(
                                              \begin{array}{cccc}
                                                1 & -1 & 1 & 0 \\
                  0 & 0 & 2 & -1 \\
                                              \end{array}
                                            \right)$ is a Gale dual of $Q$: it is a fan matrix satisfying the equivalent conditions of Proposition~\ref{prop:PWS}. As observed in Remark \ref{rems:SF}(3), $|\SF(V)|=1$, which is that $V$ determines a unique associated 2--dimensional PWS $X=X(\Si)$ of rank 2, whose fan $\Si$ is described by all the faces of the following maximal cones:
\begin{equation*}
    \Si(2)=\left\{\s_1=\left\langle\begin{array}{cc}
                                     1 & 1 \\
                                     0 & 2
                                   \end{array}
    \right\rangle, \s_2=\left\langle\begin{array}{cc}
                                     1 & -1 \\
                                     2 & 0
                                   \end{array}
    \right\rangle, \s_3=\left\langle\begin{array}{cc}
                                     -1 & 0 \\
                                     0 & -1
                                   \end{array}
    \right\rangle, \s_4=\left\langle\begin{array}{cc}
                                     0 & 1 \\
                                     -1 & 0
                                   \end{array}
    \right\rangle\right\}
\end{equation*}
Theorem \ref{thm:generazione} guarantees that the upper two rows of $U_Q$ describe the generators of the divisor class group in terms of the torus invariant divisors $D_1,\ldots,D_4$, which is
\begin{equation*}
    \Cl(X)=\Z[d(L_1)]\oplus\Z[d(L_2)]\quad\text{with}\quad L_1:=D_1\ ,\ L_2:=-D_1+D_2\,.
\end{equation*}
Notice that the column $\q_j$ of $Q$ gives the components of the torus invariant $D_j$ in terms of the generators $L_i$, namely:
\begin{equation}\label{componenti}
    D_1=L_1\,,\,D_2=L_1+L_2\,,\,D_3=L_2\,,\,D_4=2L_2\,.
\end{equation}
Since $\d_{\Si}=\lcm\left(\det(\s_i)\,|\,1\leq i\leq 4\right)=2$, the second part of Theorem \ref{thm:generazione} shows that $2L_1, 2L_2$ are certainly Cartier divisors. Actually one can say more, by solving easy $2\times 2$ systems of linear equations and obtaining the \emph{Cartier index} $c_{\Si}(D_j)$ of all the divisors $D_j$, which is \emph{the least positive integer giving a Cartier multiple of $D_j$ in $X(\Si)$}, namely:
\begin{equation*}
    c_{\Si}(D_1)=2\,,c_{\Si}(D_2)=2\,,c_{\Si}(D_3)=2\,,c_{\Si}(D_2)=1\,.
\end{equation*}
Now we want to  explicitly exhibit the immersion $\Pic(X)\hookrightarrow \Cl(X)$ by using the method presented in \S~\ref{sssez:intersezione}. With the notations of Theorem \ref{thm:generazione} we have
$$\mathcal{I}_\Sigma=\{ I_1,I_2,I_3,I_4\}, \hbox{ with }$$
$$I_1=[2,4],\quad I_2=[1,4] , \quad I_3=[1,3],\quad I_4=[2,3]   .$$
For $j=1,\ldots 4$ we compute the transverse $Q_{I_j}^{T*}=Q_{I_j}^{-1}$:
$$Q_{I_1}^{-1}=\begin{pmatrix} 1 & 0\\ -\frac 1 2  & \frac 1 2\end{pmatrix},\quad
Q_{I_2}^{-1}=\begin{pmatrix} 1 & 0 \\ 0 & \frac 1 2\end{pmatrix},\quad
Q_{I_3}^{-1}=\begin{pmatrix} 1 & 0 \\ 0 & 1\end{pmatrix},\quad
Q_{I_4}^{-1}=\begin{pmatrix} 1 & 0\\ -1  & 1\end{pmatrix},\quad$$
We obtain
$$\HNF\left(\begin{pmatrix}
Q_{I_1}^{-1}\\Q_{I_2}^{-1}\\Q_{I_3}^{-1}\\Q_{I_4}^{-1}
\end{pmatrix}\right)= \begin{pmatrix} A\\ \mathbf{0}_2\end{pmatrix}, \hbox{ with } A=\begin{pmatrix} \frac 1 2 & 0\\ 0 & \frac 1 2\end{pmatrix} $$
so that a basis of $\Pic(X)$ in $\Cl(X)\cong \Z^2$ is given by the rows of the matrix
$A^*=\begin{pmatrix} 2&0\\ 0 & 2\end{pmatrix}$. Therefore $\Pic(X)=2 \Cl(X)$. \\
A basis of the subgroup of Cartier divisors $\mathcal{C}(X)\subseteq \mathcal{W}_T(X) \cong \Z^{n+r}$ is given by the rows of the matrix
$$C=\begin{pmatrix} 2 & 0 & 0 & 0\\ 0 & 2 & 0 & 0\\
0 & 0 & 1 & 0\\ 0 & 0 & 0 & 1\end{pmatrix}\cdot U_Q=
\begin{pmatrix}
                  2 & 0 & 0 & 0 \\
                  -2 & 2 & 0 & 0 \\
                  1 & -1 & 1 & 0 \\
                  0 & 0 & 2 & -1\end{pmatrix}.$$
\end{example}

\subsection{Further remarks and future purposes}\label{ssez:finale}

In Remark \ref{rem:motivazioni}(4) we already mentioned results \cite[Lemma~2.11]{BC} and \cite[Prop.~4.7]{Conrads} showing that a $\Q$--factorial complete toric variety of rank 1 is always a quotient of WPS. Actually this fact can be generalized to every rank $r$, by replacing the role of a WPS with that of a PWS i.e.
\begin{itemize}
  \item \emph{every $\Q$--factorial complete toric variety is always a quotient of a PWS.}
\end{itemize}
Clearly when $r=1$ this fact just gives the already mentioned results of Batyrev--Cox and Conrads, since a PWS of rank 1 is a WPS. The proof of the above stated result will be one of the main purposes of our next paper \cite{RT-QUOT} to which the interested reader is referred. Probably this is the main motivation for calling a $\Q$--factorial complete toric variety satisfying one of the equivalent conditions of Prop.~\ref{prop:PWS}, a \emph{poly weighted space}.

A final observation is that
\begin{itemize}
  \item \emph{on the contrary of a WPS, a general PWS is no more the Proj of a graded algebra.}
\end{itemize}
In fact, a PWS can be even a \emph{non--projective} variety, as the following example \ref{ex:noproj} will show. Anyway, in the forthcoming paper \cite{RTclass} we will prove that, under suitable conditions on the weight matrix, a projective PWS will be birational equivalent and isomorphic in codimension 1 to a toric cover of the Proj of a suitable weighted graded algebra, by means of a finite chain of \emph{wall crossings} in the secondary fan \cite[\S~15.3]{CLS}.

\begin{example}\label{ex:noproj} Consider the following reduced weight matrix $Q$ and its Gale dual $V$
\begin{equation*}
    Q=\left(
    \begin{array}{cccccc}
      1 & 1 & 0 & 0 & 1 & 0 \\
      0 & 1 & 1 & 1 & 0 & 0 \\
      0 & 0 & 0 & 1 & 1 & 1 \\
    \end{array}
  \right),\,\quad\quad V=\left(
                            \begin{array}{cccccc}
                              1 & 0 & 0 & 0 & -1 & 1 \\
                              0 & 1 & 0 & -1 & -1 & 2 \\
                              0 & 0 & 1 & -1 & 0 & 1 \\
                            \end{array}
                          \right)
\end{equation*}
$V$ satisfies  conditions of Proposition \ref{prop:PWS}. Then $X=X(\Si)$ is a PWS for every $\Si\in\SF(V)$. Then the reader can check that:
\begin{itemize}
  \item $\left|\SF(V)\right|=8$,
  \item only 6 among those 8 distinct fans give a projective PWS: in fact the \emph{Moving Cone} $\Mov(X)$ in the secondary fan \cite[(15.1.5)]{CLS} admits only 6 chambers,
  \item in particular the fan $\Si$ defined by taking the following maximal cones
  \begin{eqnarray*}
    \Si(3)&:=&\left\{\langle\v_2,\v_4,\v_5\rangle,\langle\v_2,\v_3,\v_5\rangle,\langle\v_1,\v_4,\v_5\rangle,\langle\v_1,\v_3,\v_5\rangle,\right.\\
    &&\left.\langle\v_2,\v_4,\v_6\rangle,\langle\v_2,\v_3,\v_6\rangle,\langle\v_1,\v_4,\v_6\rangle,\langle\v_1,\v_3,\v_6\rangle\right\}
  \end{eqnarray*}
and all their faces, defines a non--projective PWS, since it does not correspond to any chamber in $\Mov(X)$.
\end{itemize}
\end{example}

\section{Linear Algebra and Gale duality}\label{sez:GaleDuality}

\subsection{$\Z$-linear Gale duality}\label{ssez:Gale}

Let $V$ be a $n\times (n+r)$ integer matrix of rank $n$. If we think $V$ as a linear application from $\Z^{n+ r}$ to $\Z^n$ then $\ker(V)$ is a lattice in $\Z^{n+ r}$, of rank $r$ and without cotorsion. We shall denote ${\mathcal G}(V)$ the \emph{Gale dual matrix of $A$}, which is an integral $r\times (n+ r)$ matrix $Q$ such that ${\mathcal L}_r(Q)=\ker(V)$; it is well-defined up to left multiplication by $\GL_r(\Z)$.
Notice that ${\mathcal L}_r(Q)={\mathcal L}_r(V)^\perp$ w.r.t. the standard inner product in $\R^{n+r}$, and that $Q\cdot V^T=0$; moreover $Q$ can be characterized by the following
universal property:
\begin{equation*} \hbox{if $A\in \mathbf{M}_r(\Z)$ is such that $A\cdot V^T=0$ then $A=\alpha\cdot Q$ for some matrix $\alpha\in\mathbf{M}_r(\Z)$.}\end{equation*}

\begin{proposition}\label{prop:Gale1} With the notation just introduced:
\begin{enumerate}
  \item ${\mathcal L}_r({\mathcal G}(V))$ does not have cotorsion in $\Z^{n+ r}$,
  \item 
  $\left(\Z^{n+ r}/{\mathcal L}_r(V)\right)_{\rm tors}\cong \left(\Z^n/{\mathcal L}_c(V)\right)_{\rm tors}$; in particular ${\mathcal L}_c({\mathcal G}(V))$ does not have cotorsion in $\Z^r$,
    \item there exists a matrix $\alpha\in \mathbf{M}_n(\Z)\cap \GL_n(\Q)$ such that  $V=\alpha\cdot {\mathcal G}({\mathcal G}(V)) $,
  \item ${\mathcal G}({\mathcal G}(V))\sim V$ if and only if ${\mathcal L}_r(V)$ does not have cotorsion in $\Z^{n+r}$.
\end{enumerate}
\end{proposition}

\begin{proof} ($1$) follows from the fact that ${\mathcal L}_r({\mathcal G}(V))=\ker(V)$ is a kernel; ($2$): by Theorem \ref{teo:SNF} there exist invertible matrices $\alpha\in\GL_n(\Z)$, $\beta\in\GL_{n+r}(\Z)$ such that $S=\alpha V \beta$ is in $SNF$; if $c_1,\ldots,c_k$ are the non zero integers on the diagonal of $S$ then $\left(\Z^{n+ r}/{\mathcal L}_r(V)\right)_{\rm tors}\cong \Z/c_1\Z\oplus\ldots\oplus \Z/c_k\Z$; on the other hand the $SNF$ of $V^T$ is $S^T$; so that the cotorsion of ${\mathcal L}_c(V)$ and that of  ${\mathcal L}_c(V^T)$ are the same. ($3$) follows from the fact that ${\mathcal L}_r(V)$ is a subgroup of ${\mathcal L}_r({\mathcal G}({\mathcal G}(V)))$ of finite index. ($4$) follows from  ($1$) and ($3$).
\end{proof}

The following results probably give the core of $\Z$-linear Gale duality:
\begin{theorem}\label{thm:Gale2} Let $V=(\v_1,\ldots,\v_{n+r})$ be a $n \times (n+r)$ matrix of rank $n$, and $Q=\mathcal{G}(V)=(\w_1,\dots,\w_{n+r})$. Setting $I:=\{i_1,\ldots,i_k\}\subseteq\{1,\ldots,n+r\}$ let $V_I$ be the submatrix of $V$ given by the columns indexed by $I$ and $Q^I$ be the submatrix of $Q$ given by the columns indexed by $\{1,\ldots,n+r\}\backslash I$, i.e.
\begin{equation}\label{Gale_complementare}
 V_I:=\left(\v_i\ |\ i\in I\right)\quad ,\quad Q^I:= \left(\w_j\ |\ j\in\{1,\ldots,n+r\}\backslash I\right)\,.
\end{equation}
Then there is a natural isomorphism
$$ \Z^{n+r-k}/\mathcal{L}_r(Q^I)\cong\mathcal{L}_c(V)/\mathcal{L}_c(V_I) \ .$$
\end{theorem}

\begin{proof}
We may assume $i_j=j$ for $j=1,\ldots,k$.   Notice that $ {\mathcal L}_r(Q^I)=\pi({\mathcal L}_r(Q))$
where $\pi$ is the projection of $\Z^{n+r}$ on the $k+1,\ldots,n+r$ components. Consider the  map $F: \Z^{n+r-k}\to {\mathcal L}_c(V)$ defined by $F(b_{k+1},\ldots,b_{n+r})=b_{k+1}{\bf v}_{k+1}+\ldots+b_{n+r}{\bf v}_{n+r}$. Since  ${\mathcal L}_c(V)/{\mathcal L}_c(V_I)$ is generated by the images of ${\bf v}_{k+1},\ldots,{\bf v}_{n+r}$, $F$ induces a surjective map $ \Z^{n+r-k}\to {\mathcal L}_c(V)/{\mathcal L}_c(V_I)$. Moreover
we have
$F(b_{k+1},\ldots,b_{n+r})\in  {\mathcal L}_c(V_I)$ if and only if there is a vector of the form ${\bf b}=(b_1,\ldots,b_k,b_{k+1},\ldots,b_n+r)\in \mathcal{L}_r(Q)$ and this happens if and only if
 $(b_{k+1},\ldots,b_{n+r})\in {\mathcal L}_r(Q^I)$;
thus $F$ induces the required isomorphism.
\end{proof}

\begin{corollary}\label{cor:Gale-on-det} Let $V$ be a $n\times(n+r)$ matrix of maximal rank $n$ and $Q=\mathcal{G}(V)$. If $|I|=n$ then $V_I$ and $Q^I$ are square matrices of order $n$ and $r$, respectively. Then
$$ [\Z^{n}\ :\ \mathcal{L}_c(V)]\ \left|\det(Q^I)\right| = \left|\det (V_I)\right|\ .$$
In particular if $\mathcal{L}_c(V)$ does not have cotorsion in $\Z^n$ then $\left|\det(Q^I)\right| = \left|\det(V_I)\right |$.
\end{corollary}

\begin{proof} We have
\begin{eqnarray*}
[\Z^n:\mathcal{L}_c(V_I)] &=& [\Z^n:\mathcal{L}_c(V)]\cdot [\mathcal{L}_c(V):\mathcal{L}_c(V_I)]\\
&=& [\Z^n:\mathcal{L}_c(V)]\cdot [\Z^{r}:\mathcal{L}_r(Q^I)],
\end{eqnarray*}
by Theorem \ref{thm:Gale2}. Since $V$ has maximal rank, $[\Z^n:\mathcal{L}_c(V)]<\infty$; then we see that $\det(Q^I)=0$ if and only if $\det(V_I)=0$; if both are non zero then the theorem follows from the fact that $[\Z^{r}:\mathcal{L}_r(Q^I)]=\left|\det(Q^I)\right|$ and $[\Z^n :\mathcal{L}_c(V_I)]=|\det(V_I)|$.
\end{proof}

\subsection{$F$--matrices and $W$--matrices}\label{ssez:F&W}\par
In this section we investigate some characterizing properties of fan matrices and their Gale duals.\\
Let $d,m,n,r $ be positive integers.

\begin{definition}\label{def:Fcompletezza&Wpositività} Let $A$ be a $d\times m$ matrix with integer entries.\\
We say that $A$ is \emph{$W$--positive} if ${\mathcal L}_r(A)$ has a basis consisting of positive vectors. \\
We say that $A$ is \emph{$F$--complete} if the cone generated by its columns is $\R^d$.
\end{definition}

Notice that $W$--positiveness and $F$--completeness are both invariant by the left action of $\GL_d(\Z)$.\\

Easy criteria for $F$--completeness and $W$--positiveness are given by the following propositions:
\begin{proposition}\cite[Theorem 3.6, ii)]{Davis}\label{prop:Scompleteness} Let $A$ be a $d\times m$ matrix with integer entries. Then $A$ is $F$--complete if and only if $\rk(A)=d$ and for each vector column ${\bf v}$ of $A$, $-{\bf v}$ is in the cone generated by the columns of $A$ different from $\v$.
\end{proposition}

\begin{proposition}\label{prop:pos} Let $A$ be a $d\times m$ matrix with integer entries such that  each column of $A$ is non zero.
Then $A$ is $W$-positive if and only if there exists a strictly positive vector in ${\mathcal L}_r(A)$.
\end{proposition}
\begin{proof} If $A$ is $W$-positive then ${\mathcal L}_r(A)$ is generated over $\Z$ by a finite set of positive vectors; by summing up all of them and using the fact that each column of $A$ is non zero we get a strictly positive vector. Conversely, let ${\bf a}$ be a strictly positive vector in ${\mathcal L}_r(A)$. By the elementary divisor theorem there exists a $\Z$-basis $({\bf a_1},{\bf a}_2,\ldots,{\bf a}_r)$ of ${\mathcal L}_r(A)$ such that ${\bf a}$ is a positive multiple of ${\bf a}_1$. Take $k\in \N$ such that ${\bf a'}_i={\bf a}_i+k{\bf a}_1$ is positive for $i=2,\ldots,r$. Then $({\bf a}_1,{\bf a'}_2,\ldots,{\bf a'}_r)$ is a positive basis of ${\mathcal L}_r(A)$.
\end{proof}

\begin{corollary}\label{cor:AGGA}
Let $A$ be a $n\times (n+r)$ matrix with integer entries of rank $n$. Then
\begin{itemize}
\item[i)] $A$ is $F$-complete if and only if ${\mathcal G}({\mathcal G}(A))$ is $F$-complete;
\item[ii)] $A$ is $W$-positive if and only if ${\mathcal G}({\mathcal G}(A))$ is $W$-positive.
\end{itemize}
\end{corollary}

\begin{proof} It follows directly from Propositions \ref{prop:Scompleteness} and \ref{prop:pos} by recalling  that ${\mathcal L}_r(A)$ is a sublattice of ${\mathcal L}_r({\mathcal G}({\mathcal G}(A)))$ of finite index.
\end{proof}

The following proposition establishes a duality between the notions of  $F$--com\-ple\-te\-ness and $W$-positiveness.

\begin{theorem}\label{thm:Wpositiveness} Let $V$ be a $n\times (n+r)$ matrix of maximal rank $n$ such that  each column of $V$ is non zero. Let $Q={\mathcal G}(V)$, so that $Q$ is an $r\times (n+r)$ integral matrix.
Then $V$ is $F$--complete ($W$--positive) if and only if $Q$ is $W$--positive ($F$--complete).
\end{theorem}
\begin{proof}
Let ${\bf v}_1,\ldots,{\bf v}_{n+r}\in \Z^n$ be the columns of $V$; then
$${\mathcal L}_r(Q)=\{(b_1,\ldots,b_{n+r})\in \Z^{n+r}\ |\  b_1{\bf v}_1+\cdots+b_{n+r}{\bf v}_{n+r}=0\}.$$
Suppose that $V$ is $F$--complete: then, by Proposition \ref{prop:Scompleteness}, for  $i=1,\ldots,n+r$ there is a positive vector ${\mathbf u_i}\in {\mathcal L}_r(Q)$ such that its $i$-th entry is not zero. By summing up over all $i$  we get a strictly positive vector  ${\mathbf b}\in{\mathcal L}_r(Q) $, so that $Q$ is $W$--positive by Proposition \ref{prop:pos}.\\
Suppose now that $V$ is $W$--positive. By multiplying by a suitable matrix in $\GL_r(\Z)$ we can suppose that all the coefficients of $V$ are in $\N$.
Let ${\bf w}_1,\ldots,{\bf w}_{n+r}\in \Z^r$ be the columns of $Q$. Since each column of $V$ is non zero, for $i=1,\ldots,n+r$ there is a row in $V$ having a non zero coefficient at place $i$; moreover ${\mathcal L}_r(V)\subseteq {\mathcal L}_r({\mathcal G}(Q))$. This means that $-{\bf w}_i$
 is in the cone generated by the vectors ${\bf w}_j, j\not =i$ and therefore $Q$ is $F$--complete.\\
 The converse assertions then follow from Corollary \ref{cor:AGGA}.
\end{proof}

\begin{definition}\label{def:Wmatrix} A \emph{$W$--matrix} is an $r\times (n+r)$ matrix $Q$  with integer entries, satisfying the following conditions:
\begin{itemize}
\item[a)] $\rk(Q)=r$;
\item[b)] ${\mathcal L}_r(Q)$ does not have cotorsion in $\Z^{n+r}$;
\item[c)] $Q$ is $W$--positive
\item[d)] Every column of $Q$ is non-zero.
\item[e)] ${\mathcal L}_r(Q)$   does not contain vectors of the form $(0,\ldots,0,1,0,\ldots,0)$.
\item[f)]  ${\mathcal L}_r(Q)$ does not contain vectors of the form $(0,a,0,\ldots,0,b,0,\ldots,0)$, with $ab<0$.
\end{itemize}
\end{definition}

Conditions equivalent to the holdness of both conditions (a) and (b), which are useful in applications,  are
\begin{itemize}
\item $\HNF(Q^T)=\left (\begin{array}{l} {\bf I}_r\\ {\bf 0}_{n,r} \end{array}\right),$
where $ {\bf I}_r$ is the identity matrix in $\M_r(\R)$ and  ${\bf 0}_{n,r}$ is the $n\times r$ zero-matrix.\\
\item $Q\sim {\mathcal G}({\mathcal G}(Q))$.
\end{itemize}

\begin{definition}\label{def:Fmatrix} An \emph{$F$--matrix} is a $n\times (n+r)$ matrix  $V$ with integer entries, satisfying the conditions:
\begin{itemize}
\item[a)] $\rk(V)=n$;
\item[b)] $V$ is $F$--complete;
\item[c)] all the columns of $V$ are non zero;
\item[d)] if ${\bf  v}$ is a column of $V$, then $V$ does not contain another column of the form $\lambda  {\bf  v}$ where $\lambda>0$ is real number.
\end{itemize}
A \emph{$CF$--matrix} is a $F$-matrix satisfying the further requirement
\begin{itemize}
\item[e)] ${\mathcal L}_c(V)=\Z^n$.
\end{itemize}
\end{definition}

\begin{proposition}\label{prop:quisopra}
$V$ is an $F$--matrix if and only if ${\mathcal G}({\mathcal G}(V))$ is a $CF$--matrix.
\end{proposition}
\begin{proof} All verifications are immediate,  applying Proposition \ref{prop:Gale1} ($3$) and Theorem \ref{thm:Wpositiveness}.
\end{proof}

 \begin{proposition}\label{prop:W-F-Gale} Let $A$ be an $n\times (n+r)$ integer matrix of rank $n$. Then
 \begin{enumerate}
 \item If $A$ is a $W$--matrix then ${\mathcal G}(A)$ is a $CF$--matrix;
 \item   $A$ is an $F$--matrix if and only if ${\mathcal G}(A)$ is a $W$--matrix.
\end{enumerate}
 \end{proposition}
 \begin{proof} $1$): suppose that $A$ is a $W$--matrix; then ${\mathcal G}(A)$ is a $r\times (n+r)$ matrix of rank $r$: moreover it is $F$--com\-ple\-te by Theorem \ref{thm:Wpositiveness}. If ${\mathcal G}(A)$ had a zero column, then ${\mathcal L}_r(A)$ should contain a vector of the form $(0,\ldots 0,1,0,\ldots,0)$ and this would contradict condition (e) of the definition of $W$--matrix. If ${\mathcal G}(A)$ had two proportional columns, then $A$ should violate condition (f).  Thus ${\mathcal G}(A)$ is an $F$-matrix. Moreover ${\mathcal L}_r({\mathcal G}(A))$ does not have cotorsion in $\Z^{n+r}$ (because so happens for every matrix in the image of ${\mathcal G}$), therefore it is a $CF$-matrix.\\
 \noindent $2$) suppose that $A$ is a $F$--matrix; then ${\mathcal G}(A)$ is a $r\times (n+r)$ matrix of rank $r$:  it is $W$--positive by Theorem \ref{thm:Wpositiveness}; it cannot contain a zero column, because $A$ is $F$--complete; since $A$ is in the image of $\mathcal{G}$, $\mathcal{L}_r(\mathcal{G}(A))$ does not have cotorsion in $\Z^{n+r}$. Moreover $\mathcal{L}_r(\mathcal{G}(A))$ must satisfy conditions e) and f) of Definition \ref{def:Wmatrix} because $A$ satisfies conditions c) and d) of Definition \ref{def:Fmatrix}. Then  $\mathcal{G}(A)$ is a $W$-matrix. For the converse apply part $1$) of this proposition and Proposition \ref{prop:quisopra}.
\end{proof}

\subsection{$W$-positive and positive matrices}\label{ssez:positiva}
Let $Q$ be a $r\times(n+r)$ $W$-matrix. We exhibit an effective procedure allowing to calculate a positive  matrix $Q'\sim Q$. Put $V=\mathcal{G}(Q)$; it is $F$-complete by Theorem \ref{thm:Wpositiveness}; therefore for $i=1,\ldots,n+r$ there is a relation
$\sum_{j=1}^{n+r}c_{ij}\mathbf{v}_j=\mathbf{0}$ such that the $c_{ij}$ are non negative integers and $c_{ii}>0$.
Such relations  are computable by finding the components of $-\mathbf{v}_i$ w.r.t. each maximal system of linearly independent columns of $V$ different from $\mathbf{v}_i$, and looking for non negative solutions. By summing up over all $i$'s and dividing by the $\gcd$ of the resulting coefficients if necessary, we get a relation $\sum_{i=1}^{n+r}c_i\mathbf{v}_i$ such that $c_i>0$ for every $i$ and $\gcd(c_1,\ldots,c_{n+r})=1$. Since $\mathcal{L}_r(Q)$ has no cotorsion,
the vector $\mathbf{c}=(c_1,\ldots,c_{n+r})\in\mathcal{L}_r(Q)$. Let $\mathbf{r}_1,\ldots,\mathbf{r}_{r}$ be the rows of $Q$. Then we can find coprime $\lambda_1,\ldots,\lambda_r\in\Z$ such that $\sum_i\lambda_i\mathbf{r}_i=\mathbf{c}$. The $\HNF$ algorithm gives a matrix $\alpha\in\GL_r(\Z)$ such that $\alpha\cdot\begin{pmatrix}
\lambda_1\\ \vdots \\\lambda_r\end{pmatrix}=\begin{pmatrix}
1\\0\\\vdots \\0\end{pmatrix}$. Then
$$\mathbf{c}=(\lambda_1,\ldots,\lambda_r)\cdot Q=(1,0,\ldots,0)\cdot (\alpha^{-1})^T\cdot Q,$$
so that $Q''=(\alpha^{-1})^T\cdot Q$ is a matrix having $\mathbf{c}$ at the first row and such that $Q''\sim Q$. By adding a suitable multiple of $\mathbf{c}$ to the other rows of $Q''$ we get a positive matrix $Q'\sim Q$.
\subsection{Reduced $F$ and $W$--matrices}

\begin{definition}\label{def:F-red}
Let $V$ be a $F$-matrix, and let $d_i$ be the $\gcd$ of the elements on the $i$-th column, for $i=1,\ldots,n+r$.
\begin{itemize}
\item We say that $V$ is $F$-\emph{reduced} if $d_i=1$ for $i=1,..,n+r$;
\item The $F$-matrix $V^{F\text{-red}}=V\cdot \diag(d_1,\ldots,d_{n+r})^{-1}$ is called the $F$--\emph{reduction} of $V$.
\end{itemize}
\end{definition}

It is clear that the cones spanned by columns of $V$ and $V^{red}$ coincide.

\begin{definition}\label{def:W-red}\hfill
\begin{itemize}
\item  A $W$-matrix $Q$ is said to be \emph{$W$-reduced} if $\mathcal{G}(Q)$ is $F$-reduced.
\item The $W$-matrix $Q^{W\text{-red}}=\mathcal{G}(\mathcal{G}(Q)^{F\text{-red}})$ is called the $W$--\emph{reduction} of $Q$.
\end{itemize}
\end{definition}

 We shall omit the superscripts $F,W$ in reductions, when they will be clear from the context.\\
 The next results provides an intrinsic criterion to verify if a $W$-matrix is reduced, and, if it is not the case, to obtain its reduction.
 \begin{theorem}[Reduction Theorem]\label{thm:riduzione}
 Let $Q$ be a $r\times (n+r)$ $W$-matrix and $V=\mathcal{G}(Q)$. For $i=1,\ldots,n+r$ denote by $Q^i$ the  matrix obtained by removing from $Q$ the $i$-th column, by $d_i$ the $\gcd$ of the elements on the $i$-th column of $V$, and by $V_{(i)}$ the matrix obtained from $V$ by dividing by $d_i$ the $i$-th column. Then
 \begin{enumerate}
 \item $Q$ is reduced if and only if for every $i=1,\ldots,n+r$, $\mathcal{L}_r(Q^i)$ has no cotorsion in $\Z^{n+r-1}$.
 \item  $\Z^{n+r-1}/\mathcal{L}_r(Q^i)$ is a cyclic group of order $d_i$.
 \item Let $\tilde{Q}^i$ be the $\SNF$ of $Q^i$, and $\alpha_i\in\GL_r(\Z)$, $\beta_i\in \GL_{n+r-1}(\Z)$ be such that $\alpha_i Q^i\beta_i=\tilde{Q}^i$.
 Let $Q_{(i)}$ be the matrix obtained from $\alpha_iQ$ by multiplying the $i$-th column by $d_i$ and  dividing the $r$-th row by $d_i$.
 Then $Q_{(i)}=\mathcal{G}(V_{(i)})$.
 \item 
 Call the matrix $Q_{(i)}$ obtained in $(3)$ the $i$-reduction of $Q$. Then $Q^{W\text{\rm{-red}}}$ can be obtained by iterating $i$-reductions, for $i=1,\ldots,n+r$.
 \end{enumerate}
 \end{theorem}
 \begin{proof}
 Of course $(1)$ follows from $(2)$ and $(4)$ follows from $(3)$. In order to prove $(2)$, apply Theorem \ref{thm:Gale2} with $I=\{i\}$; then we see that there is an isomorphism
 $$\Z^{n+r-1}/\mathcal{L}_r(Q^i)\cong \mathcal{L}_c(V)/\mathcal{L}_c(\mathbf{v}_i)\cong \Z^n/\mathcal{L}_c(\mathbf{v}_i)\cong \Z/ d_i\Z\oplus \Z^{n-1}.$$
 Now we prove $(3)$; notice that, by $(2)$, the last row of $\alpha_iQ^i$ is divisible by $d_i$, so that $Q_{(i)}$ has integer entries. Set $\gamma_i=\diag(1,\ldots,1,d_i)\in \mathbf{M}_r(\Z)$ and $\delta_i=\diag(1,\ldots,1,d_i,1,\ldots,1)\in \mathbf{M}_{n+r}(\Z)$, with $d_i$ on the  $(i,i)$--place. Then $Q_{(i)}=\gamma_i^{-1}\alpha_iQ\delta_i$ and $V_{(i)}=V\delta_i^{-1}$, so that $Q_{(i)}\cdot {V_{(i)}}^T=0$. Moreover
 $$\mathcal{L}_c(Q_{(i)})\supseteq \mathcal{L}_c(\gamma_i^{-1}\alpha_iQ^i)=\Z^n \,$$
 since the $\SNF$ of $\gamma_i^{-1}\alpha_iQ^i$ is the identity matrix. Therefore $\mathcal{L}_r(Q_{(i)})$ has no cotorsion in $\Z^{n+r}$, so that $Q_{(i)}=\mathcal{G}(V_{(i)})$ by Proposition \ref{prop:Gale1} (2).
 \end{proof}

 An alternative reduction procedure has been described in \cite{Ahmadinezhad} to which the interested reader is referred.

 \begin{example}
Let $Q$ be he $2\times 4$ $W$-matrix
$$Q:=\begin{pmatrix} 1 & 2&0&0\\
0&0&3&5
 \end{pmatrix}$$
 Then the $HNF$ of $V=\mathcal{G}(Q)$ is
 $$V:=\begin{pmatrix} 2 & -1&0&0\\
 0&0&5&-3
  \end{pmatrix}$$

With the notations of Theorem \ref{thm:riduzione} we have $d_1=2$ and
$$Q^1= \begin{pmatrix}  2&0&0\\
0&3&5 \end{pmatrix}, \quad \tilde{Q}^1=\begin{pmatrix}  1&0&0\\
0&2&0 \end{pmatrix},\quad
\alpha_1=\begin{pmatrix}
1&1\\ -1&0
\end{pmatrix},\quad
\beta_1=\begin{pmatrix} 0&-1&0\\
-3&-6&-5\\
2&4&3 \end{pmatrix} $$
so that the $1$-reduction is given by
$$Q_{(1)}=\begin{pmatrix} 2&2&3&5\\
-1&-1&0&0 \end{pmatrix}, \quad \mathcal{G}(Q_{(1)})=\begin{pmatrix} 1&-1&0&0\\ 0&0 &5 &-3\end{pmatrix}=V_{(1)}.$$
By performing successively the $3$ and $4$-reductions we obtain
$$Q_{(1,3)}=\begin{pmatrix} 2&2&15&5\\
-1&-1&-6&-2 \end{pmatrix}, \quad \mathcal{G}(Q_{(1,3)})=\begin{pmatrix} 1&-1&0&0\\ 0&0 &1 &-3\end{pmatrix}=V_{(1,3)};$$
$$Q^{W\text{\rm{-red}}}=Q_{(1,3,4)}=\begin{pmatrix} 2&2&15&15\\
-1&-1&-7&-7 \end{pmatrix}\sim \begin{pmatrix} 1&1&0&0\\
0&0&1&1  \end{pmatrix};$$ $$\mathcal{G}(Q_{(1,3,4)})=\begin{pmatrix} 1&-1&0&0\\ 0&0 &1 &-1\end{pmatrix}=V_{(1,3,4)}=V^{F\text{-red}}.$$
 \end{example}

 \begin{remark}\label{rem:WPS} In the case $r=1$, Theorem \ref{thm:riduzione} gives the well known criterion for reducing weights in weighted projective spaces, see \cite[Definition 3]{RT-WPS} and references therein. In fact the matrices $\alpha_i$ are trivial in this case and the reduction process described in part (4) of Theorem \ref{thm:riduzione} amounts to multiply on the right  the weight matrix $Q=(q_0,\ldots, q_n)$ by the diagonal matrix $\frac 1 {\prod d_j}\diag(d_0,\ldots, d_n)$. \\ By \cite[Lemma 1(c)]{RT-WPS} $d_i=\gcd(q_0,\ldots,q_{i-1},q_{i+1},\ldots,q_n)$.\\ By \cite[Prop. 3(2)]{RT-WPS} $\frac {\prod d_j}{d_i}=\lcm(d_0,\ldots,d_{i-1},d_{i+1},\ldots,d_n)=:a_i$.
 \end{remark}

\subsubsection{$W$--positiveness and row echelon form}

\begin{theorem}\label{thm:REF}
Let $A=(a_{i,j})$ be a $d\times m$ $W$--positive matrix. Then there exist  $\alpha\in \GL_d(\Z)$, a permutation matrix $\beta \in \GL_m(\Z)$ and a positive
matrix $\underline{A}$ in row echelon form  such that $\alpha A \beta =\underline{A}$.
\end{theorem}

\begin{proof} Up to the left multiplication by a matrix in $\GL_d(\Z)$ we can think of $A$ to be a positive matrix.\\
As a first step we show that:
\begin{itemize}
\item[$(i)$] if $A$ is a $d\times m$ positive matrix, with $d\geq 2$, then there exist $\alpha'\in \GL_d(\Z)$ and a permutation matrix $\beta' \in \GL_m(\Z)$ such that $\alpha' A\beta'$ is positive with a zero entry in the $(d,1)$--place.
\end{itemize}
If the $d$--th row of $A$ has some zero entry then we are done up to a permutation on columns.
 We can then assume that $A$ has a strictly positive $d$--th row. Let $\beta'\in\GL_m(\Z)$ be a permutation matrix such that $A\beta'=(a_{i,j})$ satisfies the condition
 \begin{equation}\label{ordine}
     \frac {a_{d-1,j}}{a_{d,j}}\geq \frac {a_{d-1,j+1}}{a_{d,j+1}}\ .
 \end{equation}
 Consider the last two rows of $A$, giving the submatrix
 $$\widetilde{A}:=\left ( \begin{array}{llll} a_{d-1,1}&a_{d-1,2} &\ldots &a_{d-1,m}\\
  a_{d,1}&a_{d,2} &\ldots &a_{d,m}
  \end{array}\right )$$
  Set $D=\gcd(a_{d-1,1},a_{d,1})$. The B\'ezout identity gives a matrix $$\widetilde{\alpha}:=\begin{pmatrix} x & y\\ -\frac{a_{d,1}} D & \frac{a_{d-1,1}} D\end{pmatrix} \in \SL_2(\Z)$$
  such that
  $$\widetilde{\alpha}\cdot\widetilde{A}=\left ( \begin{array}{llll} D&a'_{d-1,2} &\ldots &a'_{d-1,m}\\
  0&a'_{d,2} &\ldots &a'_{d,m}
  \end{array}\right )$$
  where, for $2\leq j\leq m$\ ,
  \begin{eqnarray*} a'_{d-1,j} &=& xa_{d-1,j}+ya_{d,j}\\
  a'_{d,j} &=& -\frac{a_{d,1}a_{d-1,j}} D+ \frac{a_{d-1,1}a_{d,j}} D
  \end{eqnarray*}
  Notice that (\ref{ordine}) gives $a'_{d,j}= \frac {a_{d,j}a_{d,1}} D\left (\frac{a_{d-1,1}}{a_{d,1}}-\frac {a_{d-1,j}}{a_{d,j}}\right )\geq 0$ and $a'_{d,j}=0$  if and only if $\frac{a_{d-1,j}}{a_{d,j}}=\frac{a_{d-1,1}}{a_{d,1}}$, since also $A\beta'$ has a strictly positive $d$--th row; in this case
  $$a'_{d-1,j}=a_{d,j}\left (x\frac{a_{d-1,j}}{a_{d,j}}+y\right )=a_{d,j}\left (x\frac{a_{d-1,1}}{a_{d,1}}+y\right )=\frac {a_{d,j}}{a_{d,1}}D > 0\ .$$
  Therefore $a'_{d-1,j}>0$ when $a'_{d,j}=0$. This means that, up to add to the $(d-1)$--st row a suitable multiple of the $d$--th row, which is up to the left multiplication of a suitable matrix $\widetilde{\alpha}'\in\GL_d(\Z)$, we can conclude the existence of a matrix $\alpha'=\widetilde{\alpha}'\cdot\begin{pmatrix} \mathbf{I}_{d-2} & \mathbf{0}\\ \mathbf{0} & \widetilde{\alpha} \end{pmatrix}\in\GL_d(\Z)$ such that $\alpha' A\beta'$ is a positive matrix having $0$ in the $(d,1)$--place, proving $(i)$.\\
Now we prove
\begin{itemize}
  \item[$(ii)$] there exists $\alpha\in \GL_d(\Z)$ and a permutation matrix $\beta \in \GL_m(\Z)$ such that $\alpha A\beta $ is positive and the first column is zero except possibly for its first entry.
 \end{itemize}
 We prove $(ii)$ by induction on $d$. The case $d=1$ is obvious.\\
 If $d>1$ we can apply $(i)$ in order to obtain
 the existence of a matrix $\alpha'\in\GL_d(\Z)$, and a permutation matrix $\beta'\in\GL_m(\Z)$ such that $A'=\alpha'\cdot A\cdot\beta'$ is a positive matrix having $0$ at place $(d,1)$. Due to relations (\ref{ordine}) and the structure of the matrix $\alpha'$, the entries $a'_{d,j}$ of the last row of $A'$ satisfy the following condition
$$
a'_{d,j}\quad\text{are}\quad\begin{array}{cc}
                  = 0 & \quad\text{for}\ 1\leq j\leq j_0 \\
                  >0 & \quad\text{for}\ j>j_0
                \end{array}
\quad\text{for some $j_0\geq 1$.}
$$
Consider now the $(d-1)\times j_0$ submatrix $\widehat{A}$ of $A'$ consisting of the first $d-1$ rows of $A'$ truncated at $j_0$. By induction there exist $\widehat{\alpha}\in \GL_{d-1}(\Z)$ and a permutation matrix $\widehat{\beta}\in\GL_{j_0}(\Z)$ such that $\widehat{\alpha}\cdot\widehat{A}\cdot\widehat{\beta}$ is positive and the first column is zero except possibly for its first entry.
By putting $\alpha''=\begin{pmatrix} \widehat{\alpha} & \mathbf{0}\\ \mathbf{0} & 1 \end{pmatrix}, \beta''=\begin{pmatrix} \widehat{\beta} & \mathbf{0}\\ \mathbf{0}& \mathbf{I}_{m-j_0}\end{pmatrix}$
we obtain that the first column of $A'':=\alpha''\cdot A'\cdot \beta''$ is zero except possibly for its first entry. Moreover the $d$-th row of $A''$ and the columns from 1 to $j_0$ of $A''$ are positive. The columns from $j_0+1$ to $m$ of $A''$ admit a strictly positive $(d-1)$--st entry. Then by summing a suitable multiple of the $(d-1)$-st row of $A''$ to the previous rows we get a positive matrix, meaning that there exists $\alpha'''\in \GL_{d}(\Z)$ such that $A''':=\alpha'''\cdot A''$ is positive. Then, by setting $\alpha=\alpha'''\alpha''\alpha'$ and $\beta=\beta'\beta''$, the matrix $A'''=\alpha\cdot A\cdot \beta$ satisfies condition $(ii)$.

\noindent

Now we prove the theorem by induction on $m$. If $m=1$ then $(ii)$ suffices to conclude the proof. By induction let us now assume that the theorem holds for a positive matrix with $m-1$ columns.  Let now $A$ be a $d\times m$ positive matrix. By $(ii)$ we may assume that the first column of $A$ is zero except possibly for its first entry $a_{1,1}$. We have now two possibilities.
\begin{itemize}
  \item $a_{1,1}=0$\ : let $A'$ be the submatrix of $A$ obtained by deleting the first column of $A$. By induction there exist $\alpha\in\GL_d(\Z)$ and a permutation matrix $\beta'\in\GL_{m-1}(\Z)$ such that $\alpha A'\beta'$ is a positive matrix in row echelon form. Then we are done by setting $\beta=\begin{pmatrix} 1 & \mathbf{0}\\ \mathbf{0} & \beta' \end{pmatrix}$.
  \item $a_{1,1}> 0$\ : let $A''$ be the submatrix of $A$ obtained by deleting the first row and the first column of $A$. By induction there exist $\alpha''\in\GL_{d-1}(\Z)$ and a permutation matrix $\beta''\in\GL_{m-1}(\Z)$ such that $\alpha'' A''\beta''$ is a positive  matrix in row echelon form. Then we end up the proof by setting
$$ \alpha=\left(
            \begin{array}{cc}
              1 & \mathbf{0} \\
              \mathbf{0}& \alpha'' \\
            \end{array}
          \right)\quad\text{and}\quad\left(
                                        \begin{array}{cc}
                                          1 & \mathbf{0} \\
                                          \mathbf{0} & \beta'' \\
                                        \end{array}
                                      \right)\ .
$$
\end{itemize}\end{proof}

\begin{remark}
The above proof exhibits a procedure giving a row echelon form of a given $W$-positive matrix $A$, provided that $A$  is already a positive matrix. When this is not the case, if $A$ is a $W$-matrix, then one can apply the procedure described in \S\  \ref{ssez:positiva} in order to find an equivalent positive matrix.

\end{remark}

\section{Proofs of geometrical results}\label{sez:proofs}
This section is devoted to proving results stated in \S\,\ref{sez:motivazioni}; let $X, \Sigma, V, Q$  be as defined therein.

\begin{remarks} \label{rem:no-torsione}\hfill
\begin{itemize}
 \item[1.] the fan matrix $V$ associated to $\Sigma$ is a \emph{reduced $F$--matrix} in the sense of definitions \ref{def:Fmatrix} and \ref{def:F-red}.
 \item[2.] The weight matrix  $Q$ is a \emph{ $W$-matrix} in the sense of Definition \ref{def:Wmatrix}.
\item[3.]  In general the sequence (\ref{HomZ-div-sequence}) is not exact on the right. Anyway, if the divisor class group $\Cl(X)$ is a free abelian group, then (\ref{div-sequence}) is a \emph{splitting} exact sequence implying the right exactness in (\ref{HomZ-div-sequence}). In this case $\ker(div^{\vee})$ is a co-torsion free subgroup meaning that the $\HNF$ of $Q^T$ is given by $\left(
                                                                 \begin{array}{c}
                                                                   {\bf I}_r \\
                                                                   \mathbf{0}_{n,r} \\
                                                                 \end{array}
                                                               \right)
  $.

  \end{itemize}
\end{remarks}

\begin{proposition}\label{prop:weight}
Let us fix $\Z$-bases of $M$ and $\mathcal{W}_T(X)$  in the sequence $(\ref{div-sequence})$. Then:
\begin{enumerate}
    \item a weight matrix $Q$ is a Gale dual of the fan matrix $V$ whose transpose $V^T$ is the representative matrix of $div$, i.e. $Q=\G(V)$,
    \item a weight matrix $Q$ is a $W$-matrix in the sense of Definition \ref{def:W-red},
  \item $Q$ is a $W$--reduced matrix if $V$ is a $F$--reduced matrix,
  \item a Gale dual $\G(Q)$ is a $CF$--matrix.
\end{enumerate}
\end{proposition}

\begin{proof} (1), (2) and (3) follow immediately by the definition of Gale duality, given in \ref{ssez:Gale}, and the Definition \ref{def:W-red}. (4) is a direct consequence of the Proposition \ref{prop:quisopra}.
\end{proof}

The following proposition provides a generalization of \cite[Proposition 5]{RT-WPS}.

\begin{proposition}\label{prop:QdaVeviceversa} Given a
fan matrix $V$ of $X$ then $Q=\G(V)$ is a reduced weight matrix of $X$ which can be obtained by the last $r$ rows of a matrix $U_V\in GL_{n+r}(\Z)$ such that $U_V\cdot V^T$ is a $\HNF$ matrix. Conversely, if $Q$ is a reduced weight matrix of $X$ and the divisor class group $\Cl(X)$ is a free abelian group then $V=\G(Q)$ is a $CF$--matrix which is a  fan matrix of $X$, given by the last $n$ rows of a matrix $U_Q\in GL_{n+r}(\Z)$ such that $U_Q\cdot Q^T$ is a $\HNF$ matrix.
\end{proposition}

\begin{proof} The first part of the statement follows by applying point (2) in Proposition \ref{prop:HNFuno} to the morphism $div^{\vee}$ in the exact sequence (\ref{HomZ-div-sequence}), whose representative matrix is $V$. Since $V$ is $F$--reduced then $Q=\G(V)$ is $W$--reduced by definition.

\noindent For the converse, if $\Cl(X)$ does not admit any torsion subgroup, then Remark \ref{rem:no-torsione}.3 applies. Again Proposition \ref{prop:HNFuno}(2) applied to the morphism $d$ in (13), whose representative matrix is $Q$, ends up the proof. In particular $V=\G(Q)$  is a $CF$--matrix (see the previous Proposition \ref{prop:weight} (4)) and it turns out to be a fan matrix of $X$: in fact, by Proposition \ref{prop:Gale1} (4), $V=\G(Q)=\G(\G(V'))=V'$ where $V'$ is the fan matrix of $X$ whose transposed matrix is the representative matrix of the morphism $div$. Moreover $Q=\G(V')$ is $W$--reduced if $V'$ is $F$--reduced, by definition. Then also $V$ is $F$--reduced.
\end{proof}

\subsection{Proof of Theorem \ref{thm:pi11=Tors}}\label{ssez:dimostrazione di pi11=Tors}
The exact sequence (\ref{div-sequence}) gives that
\begin{equation*}
    \Cl(X)\cong \left. \mathcal{W}_T(X)\right/\im (div)\cong \Z^{n+r}/\mathcal{L}_r(V)
\end{equation*}
so that $$\Tors(\Cl(X))\cong\Tors(\Z^{n+r}/\mathcal{L}_r(V) )\cong \Tors(\Z^{n}/\mathcal{L}_c(V))\cong  \Z^n/\mathcal{L}_r(T_n).$$
 The fact that $\Si$ is simplicial ensures that $N_{\Si(1)}$ is still a full sublattice of $N$,  and the rows of $T_n$ give a basis of $N_{\Sigma(1)}$. Hence  $N/N_{\Sigma(1)}\cong \Z^n/ \mathcal{L}_r(T_n)$.

\subsection{Proof of Proposition \ref{prop:PWS}}\label{ssez:dimostrazione di PWS} The equivalences  $(1)\Leftrightarrow (2)\Leftrightarrow (3)$ follow from Theorem \ref{thm:pi11=Tors} and its proof.\
The equivalence $(3)\Leftrightarrow (4)$ is a consequence of Corollary \ref{cor:minoricoprimi} below.

\begin{lemma}\label{prop:mcdminori} Let $A$ be a $n\times (n+r)$ matrix of rank $n$ with integral coefficients. Let $d_A$ be the $\gcd$ of all minors of order $n$ of $A$.  Let $\beta\in\GL_{n+r}(\Z)$ and put $A'=A\beta$. Then $d_{A'}=d_A$.
\end{lemma}
\begin{proof} Let $a_1,\ldots,a_k$ be the sequence of minors of $A$ of order $n$, and $a'_1,\ldots,a'_k$ be the corresponding sequence of minors of $A'$. We can assume that $A'$ is obtained  by applying  a step of Gauss $\Z$-reduction on the columns of $A$.
It is clear that interchanging two columns or changing the sign of a column simply modifies some signs in the sequence of minors, so that $d_{A'}=d_A$ in these cases.
Assume that $A'$ is obtained from $A$ by adding to a column (say $\mathbf{c_1}$) a multiple of another column (say $\mathbf{c_2}$), so that $\mathbf{c'_1}=\mathbf{c_1}+ b\mathbf{c_2}$ for some $b\in\Z$. Let $B$ be a square $n\times n$ submatrix of $A$, and $B'$ be the corresponding submatrix of $A'$. Suppose that $a_1=\det(B)$. Then: if  either $\mathbf{c_1}$ is not a column of $B$ or $\mathbf{c_2}$ is  a column of $B$ then $\det(B)=\det(B')$. If $\mathbf{c_1}$ is  a column of $B$ and $\mathbf{c_2}$ is  not a column of $B$ then $\det(B')=\det(B)+b\det(B'')$, where $B''$ is obtained from $B$ by replacing $\mathbf{c_1}$ by $\mathbf{c_2}$. Then $\det(B'')=\pm a_i=\pm a'_i $ for some $i\not=1$ and $a'_1=\pm a_1\pm ba'_i$ so that $\gcd(a_1,\ldots,a_k)=\gcd(a'_1,\ldots,a'_k)$.
\end{proof}

\begin{corollary}\label{cor:minoricoprimi} Let $A, d_A$ be as in Lemma \ref{prop:mcdminori}. Then
$$d_A=\left |\Tors\left (\Z^{n+r}/\mathcal{L}_r(A)\right)\right |=\left |\Z^{n}/\mathcal{L}_c(A)\right |.$$
In particular $\mathcal{L}_c(A)$ has no cotorsion in $\Z^{n}$ if and only if the minors of $A$ of order $n$ are coprime.
\end{corollary}
\begin{proof} Let $H=\HNF(A^T)=U\cdot A^T=\begin{pmatrix} T_n\\ \mathbf{0}_{r,n}\end{pmatrix}$. Then $H^T=A\cdot U^T$ and, by Lemma \ref{prop:mcdminori}, $d_A=d_{H^T}$. But $d_{H^T} =\det(T_n)= \left |\Z^{n}/\mathcal{L}_c(A)\right | $.
\end{proof}

\begin{proof}[Proof of Theorem \ref{thm:generazione}] (1)  $X$ is a PWS, meaning that $\Cl(X)\cong\Z^r$, by Proposition \ref{prop:PWS} and Theorem \ref{thm:pi11=Tors}. Recalling Remark \ref{rem:no-torsione}.3, we can then write
 \begin{equation*}
    \left(
       \begin{array}{c}
         \mathbf{I}_r \\
         \mathbf{0}_{n,r} \\
       \end{array}
     \right)=U_Q\cdot Q^T \ \Rightarrow\ Q\cdot U_Q^T =
                                                         \begin{pmatrix}
   \mathbf{I}_r  &\mathbf{0}_{r,n}
   \end{pmatrix}
 \end{equation*}
Then (\ref{generazione}) follows by recalling that $Q$ is a representative matrix of the morphism $d:\mathcal{W}_T(X)\to \Cl(X)$.

(2) Recall that, for any $k\in\N$, the Weil divisor $L=\sum_{j=1}^{n+r}a_jD_j$ is a Cartier divisor if it is locally principal, which is
\begin{equation}\label{cartier}
    \forall\,I\subset\{1,\ldots,n+r\}:\left\langle V^I\right\rangle\in\Si(n)\quad\exists\,\mathbf{m}_I\in M : \forall\,j\not\in I\ \mathbf{m}_I\cdot\v_j=a_j\,.
\end{equation}
For $I\subset\{1,\ldots,n+r\}$  define
$$\mathcal{P}^I=\{L=\sum_{j=1}^{n+r}a_jD_j\in \mathcal{W}_T(X)\ |\ \exists\,\mathbf{m}\in M : \forall\,j\not\in I\ \mathbf{m}\cdot\v_j=a_j\}.$$
Then $\mathcal{P}^I$ contains $\mathrm{Im}(div)$ and a $\Z$-basis of $\mathcal{P}^I$ is given by
$$\{D_j, j\in I\}\cup\{\sum_{k=1}^{n+r}v_{ik}D_k, i=1,\ldots ,n\},$$
where $\{v_{ik}\}$ is the $i$-th entry of $\mathbf{v}_k$. Let $\mathcal{B}_I=d(\mathcal{P}^I)$; then a basis of $\mathcal{B}_I$ in $\Cl(X)$ is  $$\{d(D_j), j\in I\}.$$ Through the fixed identification of $\Cl(X)$ with $\Z^r$, each $d(D_j)$ corresponds to the $j$-th column of the matrix $Q$. Therefore
$$\Pic(X)=\bigcap_{I\in\mathcal{I}_{\Si}} \mathcal{B}_I= \bigcap_{I\in\mathcal{I}_{\Si}} \mathcal{L}_c(Q_I).$$
(3) Observe that
$$Q\cdot C^T= Q\cdot U_Q^T\cdot \begin{pmatrix}B^T &
 \mathbf{0}_{r,n}\\  \mathbf{0}_{n,r}& \mathbf{I}_{n}\end{pmatrix}=
 \begin{pmatrix}
 \mathbf{I}_{r} & \mathbf{0}_{r,n}\end{pmatrix}\cdot  \begin{pmatrix}B^T & \mathbf{0}_{r,n}\\  \mathbf{0}_{n,r}& \mathbf{I}_{n}\end{pmatrix}=\begin{pmatrix} B^T & \mathbf{0}_{r,n}\end{pmatrix},$$
 so that $\mathcal{L}_r(C)\subseteq \mathcal{C}_T(X)$; moreover $$[\mathcal{W}_T(X): \mathcal{L}_r(C)]=\det(B)=[\Cl(X):\Pic(X)]= [\mathcal{W}_T(X): \mathcal{C}_T(X)];$$
 therefore $\mathcal{L}_r(C)= \mathcal{C}_T(X)$.\\
 (4) is a direct consequence of (2).
   \end{proof}

   \section{Acknowledgements} We would like to thank: Cinzia Casagrande for having pointed out to us references \cite{Ahmadinezhad} and \cite{BCS} of which we were not aware; Antonella Grassi for stimulating conversations  held in Turin during her last visit; the anonymous referee for his helpful suggestions and improvement of the paper.

\end{document}